\documentclass[12pt,a4paper,reqno]{amsart}

\usepackage[all,2cell,arrow,color]{xy}
\newdir{ >}{{}*!/-10pt/@{>}}
\usepackage[english]{babel}
\usepackage{mathtools}
\usepackage{stmaryrd}
\usepackage{tensor}
\numberwithin{equation}{section}
\usepackage{graphicx} 
\usepackage{rotating}
\usepackage[pagebackref,colorlinks,urlcolor=darkgreen,citecolor=darkgreen,linkcolor=darkgreen]{hyperref}
\usepackage{enumerate}
\usepackage{xspace}
\usepackage{xcolor}
\usepackage{floatpag}
\usepackage{MnSymbol}
\usepackage{bbm}
\usepackage[normalem]{ulem}
\usepackage{afterpage}
\usepackage{changepage}

\definecolor{darkgreen}{rgb}{0,0.45,0} 

  \newtheorem{proposition}{Proposition}[section]
  \newtheorem{lemma}[proposition]{Lemma}
  
  \newtheorem{theorem}[proposition]{Theorem}
  
  \theoremstyle{definition}
  \newtheorem{definition}[proposition]{Definition}

\theoremstyle{remark}
  \newtheorem{remark}[proposition]{Remark}
 
  \newcounter{c}
  
  \newcommand{\etyk}[1]{\vspace{-7.4mm}$$\begin{equation}\Label{#1}
  \addtocounter{c}{1}}
  \renewcommand{\]}{\ifnum \value{c}=1 $$\else \end{equation}\fi}
\newenvironment{amssidewaysfigure}
  {\begin{sidewaysfigure}\vspace*{.5\textwidth}\begin{minipage}{\textheight}\centering}
  {\end{minipage}\end{sidewaysfigure}}

\newcommand{\op}{^\mathsf{op}}
\newcommand{\rev}{^\mathsf{rev}}
\newcommand{\vop}{{\raisebox{-3pt}{\hspace{1pt}\rotatebox{90}{$^\mathsf{op}$}\!}}}

\begin{document}

\title{Maschke type theorems for Hopf monoids}

\author{Gabriella B\"ohm} 
\address{Wigner Research Centre for Physics, H-1525 Budapest 114,
P.O.B.\ 49, Hungary}
\email{bohm.gabriella@wigner.mta.hu}
\date{April 2020}
 
\begin{abstract}
We study integrals of Hopf monoids in duoidal endohom categories of naturally Frobenius map monoidales in monoidal bicategories. We prove two Maschke type theorems, relating the separability of the underlying monoid and comonoid, respectively, to the existence of normalized integrals. It covers the examples provided by
Hopf monoids in braided monoidal categories,
weak Hopf algebras,
Hopf algebroids over central base algebras,
Hopf monads on autonomous monoidal categories and
Hopf categories.
\end{abstract}
  
\maketitle


\section*{Introduction} \label{sec:intro}

One can find in recent literature many co-existing --- sometimes competing --- generalizations of Hopf algebra; which is itself a generalization of group algebra. They were introduced by different motivations and each of them has a more or less different set of axioms. However, they share some essential features that makes one wonder whether they all could be seen as various instances of some common unifying structure. This question was answered in \cite{BohmLack} in some extent, by showing that Hopf monoids in braided monoidal categories, weak Hopf algebras \cite{BNSz}, Hopf algebroids over a central base algebra \cite{Ravenel}, and Hopf monads on autonomous monoidal categories \cite{BruguieresVirelizier} are examples of Hopf comonads on a naturally Frobenius map monoidale $M$ in a suitable monoidal bicategory $\mathcal B$. Consequently, they also can be seen as bimonoids in the duoidal endohom category $\mathcal B(M,M)$, possessing an antipode in an appropriate sense. In \cite{span|V} also Hopf categories of \cite{BCV:HopfCat} were shown to fit this framework.

In this paper we use the term {\em Hopf monoid} for those bimonoids 
in the duoidal endohom category $\mathcal B(M,M)$ that arise from a Hopf comonad on the naturally Frobenius map monoidale $M$ in a monoidal bicategory $\mathcal B$; and hence possess an (unique) antipode in the sense of \cite[Theorem 7.2]{BohmLack}. The terminology is strictly restricted to this setting, it is not used in more general duoidal categories.

Maschke's classical theorem \cite{Maschke} sates that the group algebra $kG$ of a finite group $G$ over an arbitrary field $k$ is semisimple if and only if the characteristic of $k$ does not divide the order of $G$. A generalization to Hopf algebras is due to Larson and Sweedler. In \cite{LarsonSweedler} they proved that a Hopf algebra $H$ over a field is semisimple if and only if it possesses a normalized integral; that is, an $H$-module section of the counit.
A Hopf algebra over a field turns out to be semisimple if and only if it is separable; that is, its multiplication admits a bimodule section. This is no longer true for Hopf algebras over more general commutative base rings. In this more general situation it is the separability of a Hopf algebra that becomes equivalent to the existence of a normalized integral, see e.g. \cite{CMIZ}.
Maschke type theorems --- relating separability to the existence of normalized integrals --- were proved individually for most generalizations of Hopf algebra. For weak Hopf algebras in \cite{BNSz}, for Hopf algebroids in \cite{Hgd_int}, for Hopf monads in \cite{BruguieresVirelizier} and \cite{TuraevVirelizier}. Integrals in Hopf categories occurred recently in \cite{BFVV:LarsonSweedler}.

The aim of this paper is a unification of  all these generalizations into a single theorem, relating normalized integrals (in a suitable sense) to the separability of the constituent monoid (or coseparability of the constituent comonoid) of a Hopf monoid in the duoidal category $\mathcal B(M,M)$ for a naturally Frobenius map monoidale $M$ in some monoidal bicategory $\mathcal B$.
This setting is suitable to prove such a theorem because by \cite[Theorem 7.2]{BohmLack} the antipode is available (which is not the case for more general base monoidale).

The paper is organized as follows. In Section \ref{sec:B(M,M)} we recall the necessary background about 
Hopf comonads on a naturally Frobenius map monoidale $M$ in a monoidal bicategory $\mathcal B$. Such a gadget is interpreted as a Hopf monoid in the duoidal category 
$\mathcal B(M,M)$.
In Section \ref{sec:Maschke} integrals are defined for arbitrary bimonoids in duoidal categories. For Hopf monoids in $
\mathcal B(M,M)$, the existence of a normalized integral is sown to be equivalent to the separability of the constituent monoid.
In Section \ref{sec:DualMaschke} cointegrals are defined for arbitrary bimonoids in duoidal categories. For Hopf monoids in $
\mathcal B(M,M)$, the existence of a normalized cointegral is sown to be equivalent to the coseparability of the constituent comonoid.
In final Section \ref{sec:applications}, our main Theorems \ref{thm:Maschke} and \ref{thm:DualMaschke} are applied to each of the examples provided by Hopf monoids in braided monoidal categories,
weak Hopf algebras,
Hopf algebroids over central base algebras,
Hopf monads on autonomous monoidal categories and
Hopf categories.
While in the first four cases we re-obtain known results thereby, for Hopf categories it gives the first such theorems.

Note that the statements of our main Theorems \ref{thm:Maschke} and \ref{thm:DualMaschke} are apparently  dual of each other. However, no duality principle is known that would allow us to derive one of them from the other. They are proved independently, by basically different steps.

\subsection*{Acknowledgement} \ 
Financial support by the Hungarian National Research, Development and Innovation Office – NKFIH (grant K124138) is gratefully acknowledged. 


\section{Hopf comonads on naturally Frobenius map monoidales}
\label{sec:B(M,M)}

\subsection{Duoidal endohom category of a map monoidale} 
We begin with briefly recalling some information needed from \cite[Section  3]{BohmLack}. Then we prove some new identities for later use.

Throughout we work in a monoidal bicategory $\mathcal B$. Relying on the coherence theorem of \cite{GordonPowerStreet} --- which says that any monoidal bicategory is equivalent, as a tricategory, to a Gray monoid --- we do not denote explicitly the coherence 2-cells in $\mathcal B$ but the interchange isomorphisms. The monoidal product will be denoted by juxtaposition and $I$ stands for the monoidal unit. 
We use dots to denote the horizontal composition in $\mathcal B$.

We use the Australian term {\em monoidale} for a {\em pseudo-monoid} $(M,m,u)$ in $\mathcal B$. Its associativity and unitality iso 2-cells are denoted by $\cong$ without introducing symbols for them. A monoidale $(M,m,u)$ is a {\em map monoidale} if the 1-cells $m:MM \to M$ and $u:I \to M$ possess respective right adjoints $m^*$ and $u^*$ (with units $\eta_m:1\to m^*.m$ and $\eta_u:1 \to u^*.u$, counits $\varepsilon_m:m.m^*\to 1$ and $\varepsilon_u:u.u^* \to 1$) --- in which case $(M,m^*,u^*)$ is a comonoidale (i.e. pseudo-comonoid) in $\mathcal B$. Its left counitality iso 2-cell and its inverse, for example, are the mates of the left unitality 2-cells of the monoidale $(M,m,u)$; that is, the mutually inverse 2-cells
\begin{equation} \label{eq:lcounit*}
\xymatrix@R=8pt{
u^*1.m^*  \ar[r]^-\cong &
\underline{m.u1}.u^*1.m^*  \ar[r]^-{1.\varepsilon_u 1.1} &
m.m^* \ar[r]^-{\varepsilon_m} &
1 \\
1 \ar[r]^-{\eta_u1} &
u^*1.u1 \ar[r]^-{1.\eta_m.1} &
u^*1.m^*.\underline{m.u1} \ar[r]^-\cong &
u^*1.m^*}
\end{equation}
that will be denoted by $\cong$ for brevity too. Whenever we only put a symbol $\cong$ as a label of an arrow, we help the reader by under- or overlining the part to which some coherence iso 2-cell is applied.

For any 0-cell $M$ in $\mathcal B$, the endohom category $\mathcal B(M,M)$ is monoidal via the opposite $f \circ g:=g.f$ of the  horizontal composition of $\mathcal B$ and the monoidal unit given by the identity 1-cell $i$. If $M$ is also equipped with the structure of a map monoidale $(m,u)$, then there is a second monoidal structure on $\mathcal B(M,M)$ provided by the convolution product $f\bullet g:=m.fg.m^*$ whose unit is $j:=u.u^*$. (We omit explicitly denoting the associativity and unitality natural isomorphisms in both of these monoidal categories.) In fact, these monoidal structures combine into a duoidal structure (called a {\em 2-monoidal} structure in \cite{AguiarMahajan}), see \cite{Street:Belgian} and \cite[Section  3]{BohmLack}. As in \cite{BohmLack}, we denote its structure morphisms by
$$
\xymatrix@C=5pt@R=0pt{
&&&&& MM \ar@/^.6pc/[rrd]^-{m} \\
(i \bullet i \ar[rr]^-{\xi_0} &&
i) \ar@{=}[r]  & M \ar@/^.6pc/[rru]^-{m^*} \ar@/_1.4pc/@{=}[rrrr] 
&&
\Downarrow \varepsilon_m &&
M}
\qquad
\xymatrix@C=5pt@R=0pt{
&&&&& I \ar@/^.5pc/[rrd]^-{u} \\
(j \ar[rr]^-{\xi^0_0} &&
i)\ar@{=}[r] &
M \ar@/^.5pc/[rru]^-{u^*} \ar@/_1.4pc/@{=}[rrrr] &&
\Downarrow \varepsilon_u &&
M} 
$$
$$
\xymatrix@C=20pt@R=0pt{
(j \ar[r]^-{\xi^0} &
j\circ j )=M \ar[r]^-{u^*} &
I \ar@/^1.4pc/@{=}[rr] \ar@/_.5pc/[rd]_-u &
\Downarrow \eta_u &
I \ar[r]^-u &
M \\
&&& M \ar@/_.5pc/[ru]_-{u^*} }
$$
and for any 1-cells $a,b,c,d:M\to M$, we write $ \xi: (a\circ b) \bullet (c\circ d) \to
(a\bullet c) \circ (b\bullet d)$ for 
$$
\xymatrix@C=15pt@R=6pt{
&&&& MM \ar@/^.6pc/[rrd]^-{1c} \ar@{}[d]|-\cong \\
M \ar[r]^-{m^*} &
MM \ar[r]^-{a1} &
MM \ar[r]_-{1c} \ar@/^.6pc/[rru]^-{b1} &
MM \ar@{=}[rr] \ar@/_.7pc/[rd]_-m &
\ar@{}[d]|-{\Downarrow \eta_m} &
MM \ar[r]_-{b1} &
MM \ar[r]^-{1d} &
MM \ar[r]^-m &
M. \\
&&&& M \ar@/_.7pc/[ru]_-{m^*}}
$$

For any 1-cells $b,c:M \to M$, we introduce the 2-cells $\lambda_{b,c}: [(b\circ j) \bullet i] \circ c \to b \bullet c$ as
$$
\xymatrix@R=8pt{
&&&& MM \ar@{}[d]|-\cong \ar@/^.4pc/[rd]^-m \\
&&& M \ar@/^.4pc/[ru]^-{u1} \ar@{=}[rr] \ar@{}[rrd]|-\cong && 
M \ar@/^.4pc/[rd]^-c\\
M \ar[r]^-{m^*} &
MM \ar[r]^-{b1} &
MM \ar[r]^-{1c} \ar@/^.4pc/[ru]^-{u^*1}&
MM \ar@{=}[rr] \ar@/_.6pc/[rd]_(.4)m &
\ar@{}[d]|-{\Downarrow \eta_m} &
MM \ar[r]^-{u^*1} \ar@{}[d].|-\cong &
M \\
&&&& M \ar@/_.6pc/[ru]_(.6)*-<.4em>{_{m^*}} \ar@/_1.4pc/@{=}[rru] &&}
$$
and symmetrically, we introduce $\varrho_{b,c}:b\circ [i\bullet (j\circ c)] \to b \bullet c$ as 
$$
\xymatrix@R=8pt{
&& MM \ar@{}[d]|-\cong \ar@/^.4pc/[rd]^-{1u^*} \\
& M \ar@/^.4pc/[ru]^-{m^*} \ar@{=}[rr] \ar@{}[rrd]|-\cong && 
M \ar@/^.4pc/[rd]^-{1u} \\
M \ar[r]^-{1u} \ar@/^.4pc/[ru]^-{b} \ar@/_1.4pc/@{=}[rrd] &
MM \ar@{=}[rr] \ar@/_.6pc/[rd]_(.4)*-<.3em>{_m} \ar@{}[d]|-\cong &
\ar@{}[d]|-{\Downarrow \eta_m} &
MM \ar[r]^-{b1} &
MM \ar[r]^-{1c} &
MM \ar[r]^-{m} &
M. \\
&& M \ar@/_.6pc/[ru]_(.6)*-<.4em>{_{m^*}}  &&}
$$
They are natural in $b$ and $c$, and for all 1-cells $a:M \to M$, they are easily seen to render commutative (up-to the omitted coherence iso 2-cells like those in \eqref{eq:lcounit*}) the following diagram.
\begin{equation} \label{eq:Lem.6}
\xymatrix{
a \ar[r]^-{(\xi^0 \bullet 1) \circ 1} \ar[d]_-{1\circ (1\bullet \xi^0)} \ar@{=}[rd] &
[(j\circ j) \bullet i] \circ a \ar[d]^-{\lambda_{j,a}} \\
a\circ [i\bullet (j\circ j)] \ar[r]_-{\varrho_{a,j}} &
a}
\end{equation}

By one of the triangle identities of the adjunction $m\dashv m^*$, the diagram 
\begin{equation} \label{eq:eta-eps}
\xymatrix{
\underline{m.u1}.u^*1 \ar[rrr]^-{1.\varepsilon_u1} \ar[dd]_-\cong  
\ar[rd]^-{1.1.1.\eta_m} &&&
m \ar[ld]_-{1.\eta_m} \ar@{=}[dd] \\
& \underline{m.u1}.u^*1.m^*.m \ar[r]^-{1.\varepsilon_u1.1.1} \ar[d]^-\cong
\ar@{}[rd]|-{\eqref{eq:lcounit*}} &
m.m^*.m \ar[rd]_-{\varepsilon_m.1} \\
u^*1 \ar[r]_-{1.\eta_m} &
\underline{u^*1.m^*}.m \ar[rr]_-\cong &&
m}
\end{equation}
commutes. Using it together with one of Mac Lane's coherence triangles for the monoidale $(M,m,u)$, we infer the commutativity of the following diagram.
\begin{equation} \label{eq:eta_id}
\xymatrix@C=6pt{
m.\underline{u^*11.m^*1}.m1 \ar[ddd]_-\cong \ar@{}[rd]|-{\eqref{eq:eta-eps}} &
m.u^*11 \ar[l]_-{1.1.\eta_m1} \ar@{<-}[dd]^-\cong \ar@{=}[r]
\ar@{}[rd]|-{\textrm{(Mac Lane)}} &
m.u^*11 \ar[r]^-\cong \ar@{<-}[d]_-\cong &
u^*1.1m \ar@{<-}[d]^-\cong  \ar[r]^-{1.\eta_m.1} \ar@{}[rd]|-{\eqref{eq:eta-eps}} &
\underline{u^*1.m^*}.m.m1 \ar[ddd]^-\cong \\
&& \overline{m.u1}.\underline{m.u^*11} \ar[r]^-\cong \ar@{<-}[d]_-\cong &
\overline{m.u1}.u^*1.1m \ar[rdd]^-{1.\varepsilon_u1.1} & \\
& \underline{m.\overline{m1}}\overline{.u11}.u^*11 \ar[r]^-\cong \ar[ld]_-{1.1. \varepsilon_u 11} &
m.\overline{1m.u11}.u^*11 \ar[rrd]^-{1.1.\varepsilon_u11} \\
m.m1 \ar[rrrr]_-\cong &&&&
m.1m}
\end{equation}
From that easily follows the commutativity of the first diagram of Figure \ref{fig:Lem.5}, for all 1-cells $a,b,c:M\to M$. 

Again by one of Mac Lane's coherence triangles and one of the triangle identities of the adjunction $m\dashv m^*$, also the second diagram of Figure \ref{fig:Lem.5} commutes, for all 1-cells $a,b,c:M\to M$. 

\begin{amssidewaysfigure}
\thisfloatpagestyle{empty}
\begin{adjustwidth}{-60pt}{0pt}
\centering

\scalebox{.9}{$
\xymatrix@C=10pt@R=18pt{
&\\
&\\
&\\
m.1c.b1.\underline{m1.u11}.u^*11.a11.m^*1.m^* \ar[d]_-\cong \\
m.1c.b1.u^*11.a11.\underline{m^*1.m^*} \ar[d]_-\cong  \\
m.\underline{1c.b1.u^*11}.\underline{a11.1m^*}.m^* \ar[d]_-\cong \\
\underline{m.u^*11}.11c.1b1.1m^*.a1.m^* \ar[rr]^-{1.1.\eta_m1.1.1.1.1.1} \ar[d]_-\cong  
\ar@{}[rrrd]|-{\eqref{eq:eta_id}} &&
m.\underline{u^*11.m^*1}.m1.11c.1b1.1m^*.a1.m^* \ar[r]^-\cong &
\underline{m.m1}.11c.1b1.1m^*.a1.m^* \ar[d]^-\cong 
\\
 u^*1.1m.11c.1b1.1m^*.a1.m^* \ar[rr]_-{1.\eta_m.1.1.1.1.1.1} &&
\underline{u^*1.m^*}.m.1m.11c.1b1.1m^*.a1.m^* \ar[r]_-\cong &
m.1m.11c.1b1.1m^*.a1.m^*\\
&&& \\
m.1c.b1.\underline{m1.u11}.u^*11.a11.m^*1.m^*
\ar[dd]_-\cong \ar[r]^-{\raisebox{8pt}{${}_{1.1.1.\eta_m.1.1.1.1.1.1}$}} &
m.1c.b1.m^*.\overline{m.\underline{m1}}\underline{.u11}.u^*11.a11.m^*1.m^* \ar[r]^-\cong 
\ar@{}[rd]|-{\textrm{(Mac Lane)}} \ar[dd]^-\cong &
m.1c.b1.m^*.m.\underline{1m.u11}.u^*11.a11.\overline{m^*1.m^*} \ar[r]^-\cong  \ar[d]^-\cong &
m.1c.b1.m^*.m.\underline{1m.u11.u^*11.a11}.1m^*.m^* \ar[d]^-\cong \\
&& m.1c.b1.m^*.\underline{m.u1}.m.u^*11.a11.m^*1.m^* \ar[d]^-\cong &
m.1c.b1.m^*.m.u1.u^*1.a1.1m.1m^*.m^* \ar[dd]^-{1.1.1.1.1.1.1.1.1\varepsilon_m. 1} \\
m.1c.b1.u^*11.a11.\underline{m^*1.m^*} \ar[d]_-\cong  \ar[r]^-{1.1.1.\eta_m.1.1.1.1}  &
m.1c.b1.m^*.m.u^*11.a11.m^*1.m^* \ar@{=}[r] &
m.1c.b1.m^*.m.u^*11.a11.\underline{m^*1.m^*} \ar[d]^-\cong  & \\
\underline{m.1c.b1.u^*11}.a11.1m^*.m^* \ar[d]_-\cong &&
\underline{m.1c.b1.m^*.m.u^*11}.a11.1m^*.m^* \ar[d]^-\cong &
m.1c.b1.m^*.\underline{m.u1}.u^*1.a1.m^* \ar[d]^-\cong \\
u^*1.1m.11c.1b1.\underline{a11.1m^*}.m^* \ar[d]_-\cong &&
u^*1.1m.11c.1b1.1m^*.1m.\underline{a11.1m^*}.m^* \ar[d]^-\cong &
\underline{m.1c.b1.m^*.u^*1}.a1.m^*  \ar[d]^-\cong  \\
u^*1.1m.11c.1b1.1m^*.a1.m^* \ar[rr]^-{1.1.1.1.1\eta_m.1.1.1} 
\ar@/_3pc/@{=}[rrr]^-{\raisebox{15pt}{${}_{\textrm{adjoint}}$}} &&
u^*1.1m.11c.1b1.1m^*.1m.1m^*.a1.m^* \ar[r]^-{\raisebox{8pt}{${}_{1.1.1.1.1.1\varepsilon_m.1.1}$}} &
u^*1.1m.11c.1b1.1m^*.a1.m^* \ar[d]^-{1.\eta_m.1.1.1.1.1.1}  \\
&&& \underline{u^*1.m^*}.m.1m.11c.1b1.1m^*.a1.m^* \ar[d]^-\cong \\
 &&& m.1m.11c.1b1.1m^*.a1.m^*}$}
\caption{Proof of \eqref{eq:Lem.5}}
\label{fig:Lem.5}
\end{adjustwidth}
\end{amssidewaysfigure}

Now the lower paths of both diagrams of Figure \ref{fig:Lem.5} are the same. So we infer the commutativity of the following diagram, whose lower path is equal to the upper path of the first diagram of Figure \ref{fig:Lem.5}, and whose upper path is equal to the upper path of the second diagram of Figure \ref{fig:Lem.5}:
\begin{equation} \label{eq:Lem.5}
\xymatrix{
([(a\circ j)\bullet i] \circ b) \bullet c \ar[r]^-\xi \ar[rrd]_-{\lambda_{a,b} \bullet 1} &
[(a\circ j) \bullet i \bullet i] \circ (b\bullet c) \ar[r]^-{(1\bullet \xi_0) \circ 1} &
[(a\circ j) \bullet i ] \circ (b\bullet c) \ar[d]^-{\lambda_{a,b\bullet c}} \\
&& a\bullet b \bullet c.}
\end{equation}

A similar computation using 
\eqref{eq:eta-eps}, the triangle conditions on the adjunctions $m\dashv m^*$ and $u \dashv u^*$ together with Mac Lane's coherence triangles yields the commutativity of 
\begin{equation} \label{eq:Lem1}
\xymatrix@C=40pt{
([(a\circ j) \bullet i] \circ j) \bullet i \ar[r]^-{(1\circ \xi^0) \bullet 1} \ar[d]_-\xi &
([(a\circ j) \bullet i] \circ j \circ j) \bullet i  \ar[d]^-{(\lambda_{a,j} \circ 1)\bullet 1} \\
(a\circ j) \bullet i \bullet i \ar[r]_-{1\bullet \xi_0} &
(a \circ j) \bullet i
}
\end{equation}
for any 1-cell $a:M \to M$.

\subsection{Additional structure for naturally Frobenius map monoidales}
A map monoidale $(M,m,u)$ in a monoidal bicategory $\mathcal B$ is said to be {\em naturally Frobenius} if both $\pi$ and $\pi'$, defined as the respective 2-cells
$$
\xymatrix@R=10pt@C=20pt{
& MMM  \ar[r]^-{m1} \ar[dd]^-{1m} \ar@{}[rdd]|-\cong &
MM \ar[dd]_-m \ar@/^.6pc/@{=}[rd] \\
MM \ar@/_.6pc/@{=}[rd] \ar@/^.6pc/[ru]^-{1m^*} \ar@{}[r]|-{\Downarrow 1 \varepsilon_m}&&
\ar@{}[r]|-{\Downarrow \eta_m}&
MM \\
& MM \ar[r]_-m &
M \ar@/_.6pc/[ru]_-{m^*}}
\qquad
\xymatrix@R=10pt@C=20pt{
& MMM  \ar[r]^-{1m} \ar[dd]^-{m1} \ar@{}[rdd]|-\cong &
MM \ar[dd]_-m \ar@/^.6pc/@{=}[rd] \\
MM \ar@/_.6pc/@{=}[rd] \ar@/^.6pc/[ru]^-{m^*1} \ar@{}[r]|-{\Downarrow \varepsilon_m 1}&&
\ar@{}[r]|-{\Downarrow \eta_m}&
MM, \\
& MM \ar[r]_-m &
M \ar@/_.6pc/[ru]_-{m^*}}
$$
are invertible. Then $M$ is a self-dual object of $\mathcal B$, with unit $m^*.u:I \to MM$ and counit $u^*.m:MM\to I$. So any morphism $b:M\to M$ has a mate 
$$
b^-:=(
\xymatrix{
M \ar[r]^-{1u} &
MM \ar[r]^-{1m^*} &
MMM \ar[r]^-{1b1} &
MMM \ar[r]^-{m1} &
MM \ar[r]^-{u^*1} &
M}).
$$

For any 1-cells $b,c:M \to M$, in \cite[Section 4.5]{BohmLack} 2-cells
$$
\varphi_{b,c}:b \circ c^- \to [(b\bullet c) \circ j] \bullet i
\qquad \textrm{and} \qquad
\psi_{b,c}: b^- \circ c \to i \bullet [j \circ (b\bullet c)],
$$
natural in $b$ and $c$,
were introduced; and \cite[Lemma 4.2]{BohmLack} was proven about their compatibility with $\bullet$. 
Together with the 2-cells $\lambda$ and $\varrho$ of the previous section, they render commutative (modulo the omitted monoidal coherence isomorphisms) the diagram
\begin{equation} \label{eq:Lem.7}
\xymatrix{
a\circ b^- \circ c \ar[r]^-{\varphi_{a,b} \circ 1} \ar[d]_-{1\circ \psi_{b,c}} &
([(a\bullet b) \circ j] \bullet i) \circ c \ar[d]^-{\lambda_{a\bullet b,c}} \\
a\circ (i\bullet [j\circ (b\bullet c)]) \ar[r]_-{\varrho_{a,b\bullet c}} &
a\bullet b \bullet c}
\end{equation}
for any 1-cells $a,b,c:M\to M$. This is immediate by the explicit form of the occurring 2-cells.

In \cite[Lemma 4.3]{BohmLack} a further 2-cell 
$\vartheta_{f,g,h}:f\circ [(g\circ j) \bullet i] \circ h^-\to [(f\bullet h) \circ g \circ j] \bullet i$
was introduced naturally in any 1-cells $f,g,h:M\to M$. 

We will omit the subscripts of all morphisms $\varphi,\psi,\lambda,\varrho,\vartheta$ --- referring to the involved objects of $\mathcal B(M,M)$ --- if it may cause no confusion.

\begin{remark} \label{rem:rev}
Let us recall from \cite[Section 4.3]{BohmLack} the following duality.
If $(M,m,u)$ is a naturally Frobenius map monoidale in a monoidal bicategory $\mathcal B$, then $(M,m^*,u^*)$ is a naturally Frobenius map monoidale in the monoidal bicategory $\mathcal B^{\mathsf{op},\mathsf{rev}}$ obtained from $\mathcal B$ by formally reversing the 1-cells and taking the reversed monoidal product. Repeating the construction of Section \ref{sec:B(M,M)} with the naturally Frobenius map monoidale $(M,m^*,u^*)$ in $\mathcal B^{\mathsf{op},\mathsf{rev}}$, we arrive at the same category $\mathcal B^{\mathsf{op},\mathsf{rev}}(M,M)=\mathcal B(M,M)$ with the duoidal structure provided by the reversed products $\circ\rev$ and $\bullet\rev$.
As explained in \cite[Section 4.3]{BohmLack}, $b\mapsto b^-$
is the object map of a strong duoidal equivalence between the duoidal category $\mathcal B(M,M)$ of Section \ref{sec:B(M,M)} and the same category with the reversed monoidal structures.

In these duoidal  categories $(\mathcal B(M,M),\bullet,\circ)$ and $(\mathcal B(M,M),\bullet\rev,\circ \rev)$, the roles of the morphisms $\lambda$ and $\varrho$, as well as the roles of $\varphi$ and $\psi$ are pairwise interchanged, while $\vartheta$ has a symmetric counterpart $\kappa_{f,g,h}:f^-\circ [i\bullet(j\circ g)] \circ h \to i \bullet [j\circ g \circ (f\bullet h)]$ for all 1-cells $f,g,h:M\to M$. 

Recall a further duality of duoidal categories. Interchanging the roles of the $\circ$- and $\bullet$-monoidal structures we obtain a duoidal structure on the opposite category. In this way, for a (not necessarily naturally Frobenius) map monoidale $M$ in a monoidal bicategory $\mathcal B$, also $(\mathcal B(M,M)^{\mathsf{op}},\circ,\bullet)$ is a duoidal category. However, this latter one {\em does not belong to the class} described in Section \ref{sec:B(M,M)}. Therefore the results proved in $(\mathcal B(M,M),\bullet,\circ)$, allow for no straightforward dualization to $(\mathcal B(M,M)^{\mathsf{op}},\circ,\bullet)$.
\end{remark}

Using that $\xi^0: j\to j\circ j$ is induced by the unit $\eta^u$ of the adjunction $u\dashv u^*$ (see Section \ref{sec:B(M,M)}), we see that for any 1-cell $a:M\to M$,
a morphism $\theta:j\circ a\to j$ is left $j$-colinear; that is, 
$$
\xymatrix{
j\circ a \ar[r]^-\theta \ar[d]_-{\xi^0 \circ 1} &
j \ar[d]^-{\xi^0} \\
j\circ j\circ a \ar[r]_-{1\circ \theta} &
j\circ j}
$$
commutes, if and only if $\theta$ is equal to 
$$
\xymatrix@R=10pt{
&&&& M \ar@/^.8pc/[rrdd]^-{a} \\
M \ar[r]^-{u^*} &
I \ar@{=}[rr] \ar@/_.4pc/[rd]_(.3)u &
\ar@{}[d]|-{\Downarrow {\eta_u}} &
I \ar@/^.4pc/[ru]^-u \ar@{}[rr]|-{\Downarrow {\theta}} &&
 \\
&& M \ar[rr]^(.7){u^*} \ar@/_.4pc/[ru]_(.7)*-<4pt>{_{u^*}} 
\ar@{=}@/_2pc/[rrrr] &&
I \ar[rr]^-u \ar@{}[d]|(.65){\Downarrow {\varepsilon_u}} &&
M.\\
&&&&}
$$
By the explicit form of $\kappa$ in Remark \ref{rem:rev} (see the dual of $\vartheta$ in \cite[Lemma 4.3]{BohmLack}), a morphism $\theta$ of this form renders commutative

\begin{equation} \label{eq:step4}
\xymatrix{
f^- \circ [i\bullet (j\circ a)] \circ h \ar[r]^-{\kappa_{f,a,h}} \ar[d]_-{1\circ (1\bullet \theta) \circ 1} &
i\bullet [j\circ a \circ (f\bullet h)] \ar[d]^-{1\bullet (\theta \circ 1)} \\
f^-\circ h \ar[r]_-{\kappa_{f,i,h} = \psi_{f,h}} &
 i\bullet [j\circ (f\bullet h)] }
\end{equation}
for any 1-cells $f,h:M\to M$. 
(In the horizontal composite 2-cells of the columns of the diagram of \eqref{eq:step4}, only those components are non-identity 2-cells which meet identity 2-cells of the 2-cells of the rows, when they are vertically composed.)
The equality of the morphisms in the bottom row of \eqref{eq:step4} follows by the dual version of \cite[Lemma 4.3~(ii)]{BohmLack}.

For any objects $b$ and $c$, and any morphism of right $i$-modules $\varpi:i\to [(b \bullet c)\circ j] \bullet i$ in $\mathcal B(M,M)$ --- that is, such that
\begin{equation} \label{eq:ilin}
\xymatrix{
i\bullet i \ar[r]^-{\varpi \bullet 1} \ar[d]_-{\xi_0} &
[(b \bullet c)\circ j] \bullet i \bullet i \ar[d]^-{1\bullet \xi_0} \\
i\ar[r]_-\varpi & 
[(b \bullet c)\circ j] \bullet i}
\end{equation}
commutes --- also the following diagram commutes.
\begin{equation}\label{eq:Fig2}
\xymatrix@C=30pt{
i  \ar@{=}[d] \ar@{=}[r] \ar@/_4.5pc/@{=}[dddd] &
(i\circ j) \bullet i \ar[r]^-{(1\circ \xi^0) \bullet 1} &
(i\circ j\circ j) \bullet i \ar[d]^-{(\varpi \circ 1 \circ 1) \bullet 1} \\
(i\circ j) \bullet (i\circ i) \ar[r]^-{(\varpi \circ 1) \bullet 1} \ar[d]^-\xi &
[([(b \bullet c)\circ j] \bullet i) \circ j] \bullet (i \circ i) \ar[d]^-\xi \ar[r]^-{(1\circ \xi^0) \bullet 1} &
[([(b \bullet c)\circ j] \bullet i) \circ j \circ j] \bullet i \ar[ddd]^-{(\lambda \circ 1) \bullet 1} \\
(i\bullet i) \circ (j\bullet i) \ar@{=}[d] &
([(b \bullet c)\circ j] \bullet i \bullet i) \circ (j\bullet i) \ar@{=}[d] \ar@{}[r]|-{\eqref{eq:Lem1}} & \\
i \bullet  i \ar[r]^-{\varpi \bullet 1} \ar[d]^-{\xi_0}  \ar@{}[rd]|-{\eqref{eq:ilin}} &
[(b \bullet c)\circ j] \bullet i \bullet i \ar[rd]^-{1\bullet \xi_0} \\
i \ar[rr]_-{\varpi} &&
[(b \bullet c)\circ j] \bullet i }
\end{equation}
With this identity at hand, we see that for any right $i$-module morphism $\omega:i \to (b\circ c^-) \bullet i$, and the associated morphism $\Omega$ of
\begin{equation}  \label{eq:Omega}
\xymatrix@C=23pt{
j \ar[r]^-{\omega \circ 1} &
[(b\circ c^-) \bullet i] \circ j \ar[r]^-{(\varphi \bullet 1) \circ 1} &
([(b \bullet c)\circ j] \bullet i \bullet i) \circ j \ar[r]^-{(1\bullet \xi_0) \circ 1} &
([(b \bullet c)\circ j] \bullet i) \circ j  \ar[r]^-\lambda &
b\bullet c}
\end{equation}
the diagram of Figure \ref{fig:Lem3} commutes. The region marked by \eqref{eq:Fig2} commutes by the application of \eqref{eq:Fig2} to the following particular right $i$-module morphism as $\varpi$:
$$
\xymatrix{
i\ar[r]^-{\omega} &
(b\circ c^-) \bullet i\ar[r]^-{\varphi \bullet 1} &
[(b \bullet c)\circ j] \bullet i \bullet i \ar[r]^-{1\bullet \xi_0} &
[(b \bullet c)\circ j] \bullet i.}
$$
\begin{amssidewaysfigure}
\thisfloatpagestyle{empty}
{\color{white} .}
\hspace*{-2cm}
\vspace{2cm}
\centering
\scalebox{1}{$
\xymatrix@C=10pt@R=25pt{
(a\circ d^-) \bullet i\ar[rr]_-{1\bullet \xi^0 \bullet 1} \ar@/^1.5pc/@{=}[rrr]\ar[dd]_-{1\bullet \omega} 
\ar@{}[rrdd]|-{\eqref{eq:Fig2}} &&
(a\circ d^-) \bullet (j\circ j) \bullet i \ar[r]_-{\xi \bullet 1} \ar[dd]^-{1\bullet (\Omega \circ 1) \bullet 1} &
(a\circ d^-) \bullet i  \ar[r]^-{\varphi \bullet 1} \ar[ddd]^-{[(1\bullet \Omega) \circ 1] \bullet 1} &
[(a\bullet d) \circ j] \bullet i \bullet i \ar[r]^-{1\bullet \xi_0} 
\ar[dddd]^-{[(1\bullet \Omega \bullet 1) \circ 1] \bullet 1 \bullet 1} &
[(a\bullet d) \circ j] \bullet i  \ar[ddddd]^-{[(1\bullet \Omega \bullet 1) \circ 1] \bullet 1} \\
\\
(a\circ d^-) \bullet (b\circ c^-) \bullet i \ar[d]_-{\xi \bullet 1} 
\ar[r]^-{\raisebox{8pt}{${}_{1\bullet \varphi \bullet 1}$}} 
\ar@{}[rddd]|-{\mbox{\scriptsize{\cite[Lemma 4.2]{BohmLack}}}} &
(a\circ d^-) \bullet [(b\bullet c) \circ j] \bullet i \bullet i \ar[d]^-{\xi \bullet 1 \bullet 1} 
\ar[r]^-{\raisebox{8pt}{${}_{1\bullet 1 \bullet \xi_0}$}} &
(a\circ d^-) \bullet [(b\bullet c) \circ j] \bullet i  \ar[rd]^-{\xi \bullet 1} \\
[(a\bullet b) \circ (d^- \bullet c^-)] \bullet i \ar[dd]_-\cong &
[(a\bullet b \bullet c) \circ d^-] \bullet i \bullet i \ar[d]^-{\varphi \bullet 1 \bullet 1} &&
[(a\bullet b \bullet c) \circ d^-] \bullet i  \ar[rd]^-{\varphi \bullet 1} \\
& [(a\bullet b \bullet c \bullet d) \circ j] \bullet i  \bullet i  \bullet i  \ar[rrr]^-{1\bullet 1\bullet \xi_0} 
\ar[d]^-{1 \bullet \xi_0 \bullet 1} &&&
[(a\bullet b \bullet c \bullet d) \circ j] \bullet i  \bullet i \ar[rd]^-{1\bullet \xi_0} \\
[(a\bullet b) \circ (c \bullet d)^-] \bullet i \ar[r]_-{\varphi \bullet 1} &
[(a\bullet b \bullet c \bullet d) \circ j] \bullet i  \bullet i \ar[rrrr]_-{1\bullet \xi_0} &&&&
[(a\bullet b \bullet c \bullet d) \circ j] \bullet i}$}
\caption{Properties of $\Omega$ of \eqref{eq:Omega}}
\label{fig:Lem3}
\end{amssidewaysfigure}

\subsection{Hopf comonads}
As observed in \cite{Street:Belgian} and \cite[Section 3.3]{BohmLack}, a monoidal comonad $a$ on a map monoidale $M$ in a monoidal bicategory $\mathcal B$ can equivalently be seen as a bimonoid --- in the sense of \cite[Definition 6.25]{AguiarMahajan} --- in the duoidal category $\mathcal B(M,M)$ of Section \ref{sec:B(M,M)}. We denote by $\mu:a \bullet a \to a$ and $\eta: j\to a$ its $\bullet$-monoid structure and by $\delta:a \to a \circ a$ and $\varepsilon:a \to i$ its $\circ$-comonoid structure.

By \cite[Theorem 7.2]{BohmLack} a monoidal comonad $a$ on a naturally Frobenius map monoidale $M$ is a Hopf comonad (a right Hopf comonad in the terminology of \cite{ChikhLackStreet}) if and only if there is a 2-cell $\sigma: a \to a^-$ --- the so-called {\em antipode} --- rendering commutative the diagrams of \cite[Theorem 7.2]{BohmLack}. In this case we term $a$ --- with the $\bullet$-monoid structure $(\mu,\eta)$, $\circ$-comonoid structure $(\delta,\varepsilon)$, and the antipode $\sigma:a \to a^-$ --- a {\em Hopf monoid} in $\mathcal B(M,M)$. For a Hopf monoid $a$, also the following diagram commutes.
\begin{equation} \label{eq:step3}
\xymatrix{
a \ar[r]^-\delta \ar[d]^-\delta \ar@/_1.7pc/@{=}[ddd] &
a\circ a \ar[r]^-{1\circ \sigma} \ar[d]^-{1\circ \delta} \ar@{}[rrd]|-{\mbox{\scriptsize \cite[Theorem 7.5]{BohmLack}}} &
a \circ a^- \ar[r]^-{1\circ \delta^-} &
a \circ (a\circ a)^-\ar[d]^-\cong \\
a\circ a \ar[r]^-{\delta \circ  1} \ar[dd]^-{\varepsilon \circ  1} \ar@{}[rrdd]|-{\mbox{\scriptsize \cite[Theorem 7.2]{BohmLack}}} &
a \circ a \circ a \ar[r]^-{1\circ \sigma \circ 1} &
a \circ a^- \circ a \ar[r]^-{ 1 \circ 1 \circ \sigma} \ar[d]^-{\varphi \circ 1} &
a \circ a^- \circ a^-  \ar[d]^-{\varphi \circ 1}  \\
&& ([(a\bullet a)\circ j]\bullet i) \circ  a \ar[d]^-{[( \mu \circ 1)\bullet 1]\circ 1} &
([(a\bullet a)\circ j]\bullet i) \circ  a^- \ar[dd]^-{[( \mu \circ 1)\bullet 1]\circ 1} \\
a \ar[r]^-{(\xi^0 \bullet  1)\circ 1} \ar[d]_-\sigma &
[(j\circ j)\bullet i]  \circ a \ar[r]^-{[(\eta \circ 1)\bullet 1]\circ 1} &
[(a\circ j)\bullet i]  \circ a \ar[rd]^-{1 \circ \sigma} \\
a^- \ar[r]_-{(\xi^0 \bullet  1)\circ 1} &
[(j\circ j)\bullet i]  \circ a^- \ar[rr]_-{[(\eta \circ 1)\bullet 1]\circ 1} &&
[(a\circ j)\bullet i]  \circ a^-}
\end{equation}

\begin{remark} \label{rem:rev_Hopf}
A monoidal comonad $a$ on a naturally Frobenius map monoidale $(M,\! m,\! u)$ in a monoidal bicategory $\mathcal B$ can be seen, equivalently, as a monoidal comonad on the naturally Frobenius map monoidale $(M,m^*,u^*)$ of Remark \ref{rem:rev} in $\mathcal B^{\mathsf{op},\mathsf{rev}}$. The comultiplications and the counits are the same while the monoidal structures are mates under the adjunctions $m\dashv m^*$ and $u\dashv u^*$ in $\mathcal B$. Hence it can equivalently be seen as a bimonoid in either one of the duoidal categories $(\mathcal B(M,M),\bullet,\circ)$ and $(\mathcal B(M,M),\bullet\rev,\circ\rev)$ of Remark \ref{rem:rev}.

The diagrams of \cite[Theorem 7.2]{BohmLack} in these duoidal categories take the same form (only their roles are interchanged). Thus they commute in $(\mathcal B(M,M),\circ,\bullet)$ if and only if they commute in $(\mathcal B(M,M),\circ\rev,\bullet\rev)$.
That is to say, $a$ is a Hopf monoid in $(\mathcal B(M,M),\circ,\bullet)$ if and only if it is a Hopf monoid in $(\mathcal B(M,M),\circ\rev,\bullet\rev)$ via the same structure morphisms.
\end{remark}


\section{Separability of a Hopf monoid} \label{sec:Maschke}

Recall that a monoid $(a,\mu,\eta)$ in a monoidal category $(\mathsf C, \bullet,j)$ is said to be {\em separable} if there is a morphism $\nabla: a\to a \bullet a$ rendering commutative the following diagrams. 
$$
\xymatrix@R=15pt{
a\bullet a \ar[rr]^-{1\bullet \nabla} \ar[dd]_-{\nabla\bullet 1} \ar[rd]_-\mu &&
a\bullet a\bullet a \ar[dd]^-{\mu \bullet 1} \\
& a \ar[rd]^-\nabla \\
a\bullet a\bullet a \ar[rr]_-{1 \bullet \mu} &&
a\bullet a}
\qquad \qquad
\xymatrix@R=42pt{
a \ar[r]^-\nabla \ar@{=}[rd] &
a\bullet a \ar[d]^-\mu \\
& a}
$$
That is, $\nabla$ is an $a$-bimodule section of $\mu$.
Although such a section $\nabla$ is not unique, in our considerations only its existence plays any role.

The aim of this section is to find sufficient and necessary conditions for the separability of the constituent monoid of a Hopf monoid.

\begin{definition} \label{def:integral}
By a {\em left integral} for a bimonoid $(a,(\mu,\eta),(\delta,\varepsilon))$ in a duoidal category $(\mathsf D,\bullet,\circ)$ we mean a left $a$-module and right $i$-module morphism $\theta:i \to a\bullet i$. That is, a morphism $\theta$ making the first two diagrams of
\begin{equation} \label{eq:int}
\xymatrix{
a \bullet i \ar[r]^-{1\bullet \theta} \ar[d]_-{\varepsilon \bullet 1} &
a\bullet a \bullet i \ar[dd]^-{\mu \bullet 1} \\
i\bullet i \ar[d]_-{\xi_0} \\
i \ar[r]_-\theta & a\bullet i }
\qquad \qquad
\xymatrix@R=63pt{
i\bullet i \ar[r]^-{\theta \bullet 1} \ar[d]_-{\xi_0} &
a \bullet i \bullet i \ar[d]^-{1\bullet \xi_0} \\
i \ar[r]_-\theta &
a\bullet i}
\qquad \qquad
\xymatrix{
i \ar[r]^-\theta \ar@{=}[rdd] &
a\bullet i \ar[d]^-{\varepsilon \bullet 1} \\
& i\bullet i \ar[d]^-{\xi_0}  \\
& i}
\end{equation}
commute. A left integral $\theta$ is said to be {\em normalized} if also the third diagram of \eqref{eq:int} commutes.

Dually, a {\em right integral} for a bimonoid $a$ in $(\mathsf D,\bullet,\circ)$ is a left integral for $a$ regarded as a bimonoid in $(\mathsf D,\bullet\rev,\circ\rev)$. That is, a right $a$-module and left $i$-module morphism $\theta:i \to i \bullet a$. A right integral is {\em normalized} if it is normalized as a left integral; that is, it is a section of the $a$-action
$\xymatrix@C=15pt{
i\bullet a \ar[r]^-{1\bullet \varepsilon} & i \bullet i \ar[r]^-{\xi_0} & i}$.
\end{definition}

\begin{lemma} \label{lem:Omega}
Consider a naturally Frobenius map monoidale $M$ in a monoidal bicategory $\mathcal B$.
For a Hopf monoid 
$(a,(\mu,\eta),(\delta,\varepsilon),\sigma)$ 
in the duoidal category $\mathcal B(M,M)$ of Section \ref{sec:B(M,M)}, and any morphism 
$\xymatrix@C=12pt{i \ar[r]^-\theta & a\bullet i}$, 
take
\begin{equation} \label{eq:omega}
\omega:=(\xymatrix{
i \ar[r]^-\theta &
a\bullet i \ar[r]^-{\delta \bullet 1} &
(a \circ a) \bullet i \ar[r]^-{(1\circ \sigma)\bullet 1} &
(a \circ a^- )\bullet i}).
\end{equation}
For the corresponding morphism $\Omega$ of \eqref{eq:Omega} the following assertions hold.
\begin{itemize}
\item[{(1)}] $\Omega$ renders commutative the diagram of Figure \ref{fig:Lem4}.
\item[{(2)}] If $\theta$ renders commutative the rightmost diagram of \eqref{eq:int} then $\mu.\Omega=\eta$.
\item[{(3)}] If $\theta$ is a left integral, then $\Omega$ renders commutative the diagram of Figure \ref{fig:Lem5}.
\end{itemize}
\end{lemma}

\begin{proof}
Part (1) is proved by Figure \ref{fig:Lem4}, part (2) is proved by Figure \ref{fig:normalized} and part (3) is proved by Figure \ref{fig:Lem5} (whose region $(\ast)$ commutes by a bimonoid axiom).
\end{proof}

\begin{figure}[p]
\hspace*{-1.3cm}
\begin{minipage}{\textwidth+2cm}
\centering
\scalebox{1}{$
\xymatrix@C=10pt@R=25pt{
a\circ a \ar[r]^-{1\circ \sigma} \ar[d]_-{1\bullet \xi^0_0} &
a\circ a^- \ar[r]^-\varphi \ar[d]^-{1\bullet \xi^0_0} &
[(a\bullet a) \circ j] \bullet i \ar[d]_-{1\bullet 1\bullet \xi^0_0} \ar@{=}@/^1.2pc/[rd] \\
(a\circ a) \bullet i \ar[d]_-{1\bullet \theta} &
(a\circ a^-) \bullet i \ar[d]^-{1\bullet \theta} \ar[r]^-{\varphi \bullet 1} 
\ar@{}[rrddddd]|-{\textrm{Figure}~\ref{fig:Lem3}} \ar@/^5pc/[ddd]^-{1\bullet \omega}  &
[(a\bullet a) \circ j] \bullet i \bullet i  \ar[r]^-{1\bullet \xi_0} &
[(a\bullet a) \circ j] \bullet i  \ar[ddddd]^-{[(1\bullet \Omega \bullet 1) \circ 1] \bullet 1} \\
(a\circ a) \bullet a \bullet i \ar[d]_-{1\bullet \delta \bullet 1} &
(a\circ a^-) \bullet a \bullet i \ar[d]_-{1\bullet \delta \bullet 1}  \\
(a\circ a) \bullet (a\circ a) \bullet i \ar[dd]_-{\xi \bullet 1} &
(a\circ a^-) \bullet (a\circ a) \bullet i \ar[d]_-{1\bullet (1\circ \sigma) \bullet 1} \\
& (a\circ a^-) \bullet (a\circ a^-) \bullet i \ar[d]^-{\xi \bullet 1} \\
[(a\bullet a) \circ (a\bullet a)] \bullet i \ar[dd]_-{(\mu \circ \mu)\bullet 1} 
\ar[r]^-{\raisebox{8pt}{${}_{[1\circ (\sigma \bullet \sigma)] \bullet 1}$}} &
[(a\bullet a) \circ (a^-\bullet a^-)] \bullet i \ar[d]^-\cong \\
\ar@{}[rd]|-{\mbox{\tiny \cite[Theorem 7.5]{BohmLack}}} &
[(a\bullet a) \circ (a\bullet a)^-] \bullet i \ar[d]^-{(\mu \circ \mu^-)\bullet 1} 
\ar[r]^-{\varphi \bullet 1} &
[(a\bullet a \bullet a \bullet a) \circ j] \bullet i \bullet i \ar[r]^-{1\bullet \xi_0} &
[(a\bullet a \bullet a \bullet a) \circ j] \bullet i \ar[d]^-{[(\mu \bullet \mu)\circ 1]\bullet 1} \\
(a\circ a) \bullet i \ar[r]_-{(1\circ \sigma) \bullet 1} &
(a\circ a^-) \bullet i \ar[r]_-{\varphi \bullet 1} &
[(a\bullet a) \circ j] \bullet i \bullet i \ar[r]_-{1\bullet \xi_0} &
[(a\bullet a) \circ j] \bullet i}$ }
\caption{Proof of Lemma \ref{lem:Omega}~(1)}
\label{fig:Lem4}
\end{minipage}
\hspace*{-1cm}
\end{figure}
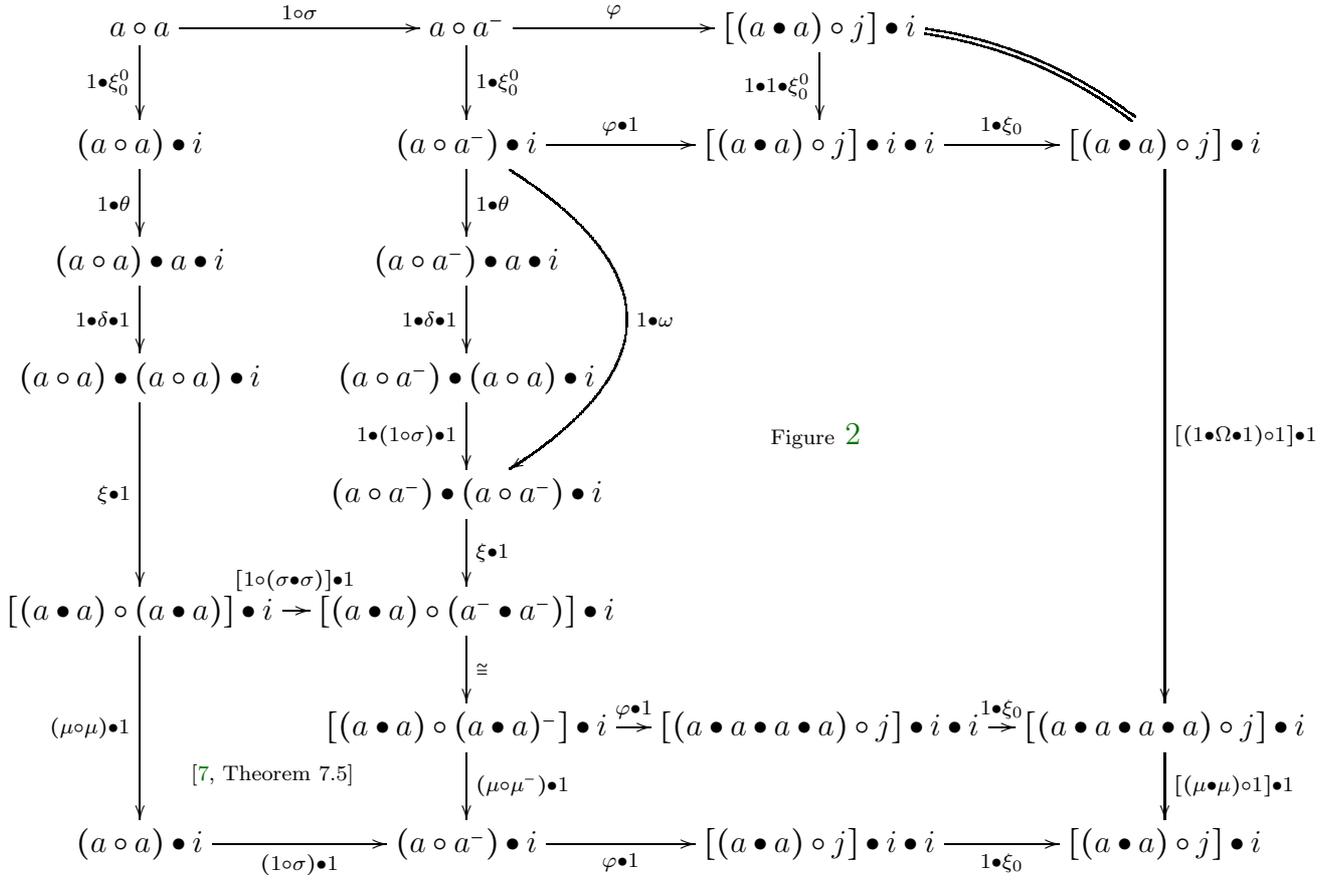

\begin{amssidewaysfigure}
\thisfloatpagestyle{empty}
{\color{white} .}
\hspace*{-2cm}
\vspace{2cm}
\centering
\scalebox{1}{$
\hspace*{1cm}
\xymatrix@C=30pt@R=25pt{
j \ar[r]_-{\theta \circ 1} \ar@/^4pc/[rrrrrr]^-\Omega \ar@{=}@/_2pc/[rdd] &
(a\bullet i) \circ j \ar[r]_-{(\delta \bullet 1) \circ 1} \ar[d]^-{(\varepsilon \bullet 1) \circ 1}
\ar@{}[rrrd]|-{\mbox{\tiny \cite[Theorem 7.2]{BohmLack}}} &
[(a\circ a)\bullet i] \circ j  \ar[r]_-{[(1\circ \sigma) \bullet 1] \circ 1} &
[(a\circ a^-)\bullet i] \circ j  \ar[r]^-{(\varphi \bullet 1) \circ 1} &
([(a\bullet a) \circ j] \bullet i \bullet i) \circ j \ar[r]^-{(1\bullet \xi_0) \circ 1}
\ar[d]^-{[(\mu \circ 1) \bullet 1 \bullet 1] \circ 1} &
([(a\bullet a) \circ j] \bullet i) \circ j \ar[r]_-\lambda &
a\bullet a \ar[dd]^-\mu \\
& (i\bullet i) \circ j  \ar[r]^-{(\xi^0 \bullet 1 \bullet 1) \circ 1} \ar[d]^-{\xi_0 \circ 1} &
[(j \circ j)\bullet i\bullet i] \circ j \ar[d]^-{(1\bullet \xi_0) \circ 1} \ar[rr]^-{[(\eta \circ 1) \bullet 1 \bullet 1] \circ 1} &&
[(a \circ j)\bullet i\bullet i] \circ j \ar[r]^-{(1\bullet \xi_0) \circ 1} &
[(a \circ j)\bullet i] \circ j \ar[rd]^-\lambda \\
& j \ar[r]^-{(\xi^0 \bullet 1) \circ 1} \ar@{=}@/_2pc/[rrr] &
[(j\circ j) \bullet i] \circ j  \ar[rr]^-\lambda \ar@{}[rd]|(.4){\eqref{eq:Lem.6}} 
\ar[rrru]^-{[(\eta\circ 1) \bullet 1] \circ 1} &&
j \ar[rr]_-\eta &&
a\\
&&&&}$}
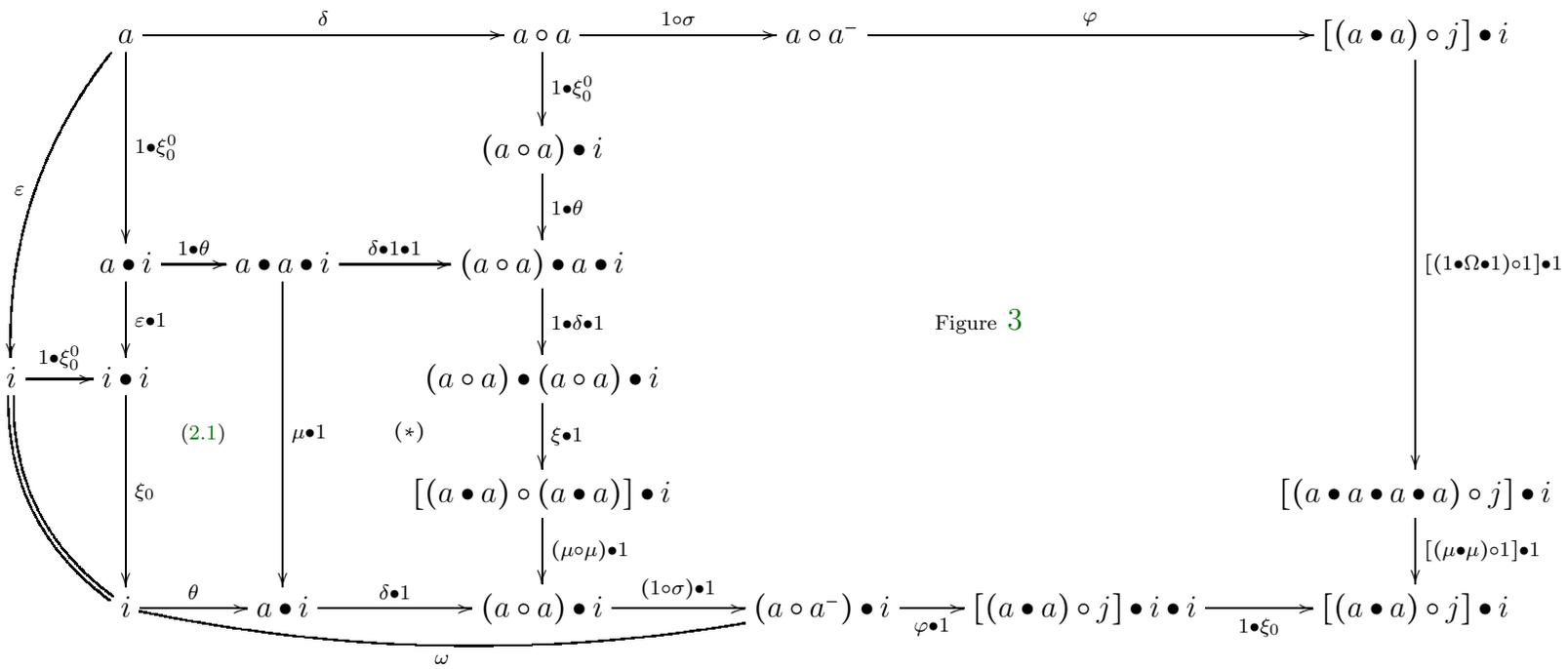
\caption{Proof of Lemma \ref{lem:Omega}~(2)}
\label{fig:normalized}
\end{amssidewaysfigure}

\begin{amssidewaysfigure}
\thisfloatpagestyle{empty}
{\color{white} .}
\hspace*{-2cm}
\vspace{2cm}
\centering
\scalebox{1}{$
\xymatrix@C=25pt@R=25pt{
& a \ar[rr]^-\delta \ar@/_1.2pc/[lddd]_-\varepsilon \ar[dd]^-{1\bullet \xi^0_0} &&
a\circ a \ar[r]^-{1\circ \sigma} \ar[d]^-{1\bullet \xi^0_0} &
a\circ a^- \ar[rr]^-\varphi &&
[(a \bullet a) \circ j] \bullet i \ar[dddd]^-{[(1\bullet \Omega \bullet 1) \circ 1] \bullet 1} \\
&&& (a\circ a) \bullet i \ar[d]^-{1\bullet \theta} 
\ar@{}[rrrddd]|-{\textrm{Figure}~\ref{fig:Lem4}} \\
& a \bullet i  \ar[r]^-{1\bullet \theta}  \ar[d]^-{\varepsilon \bullet 1}
\ar@{}[rddd]|-{\eqref{eq:int}} &
a \bullet a \bullet i  \ar[r]^-{\delta\bullet 1 \bullet 1} \ar[ddd]^-{\mu \bullet 1}
\ar@{}[rddd]|-{(\ast)} &
(a\circ a) \bullet a \bullet i \ar[d]^-{1\bullet \delta \bullet 1} \\
i \ar[r]^-{1\bullet \xi^0_0} \ar@{=}@/_1.2pc/[rdd] &
i\bullet i\ar[dd]^-{\xi_0} &&
(a\circ a) \bullet (a\circ a) \bullet i \ar[d]^-{\xi \bullet 1} \\
&&& [(a\bullet a) \circ (a\bullet a)] \bullet i \ar[d]^-{(\mu \circ \mu) \bullet 1} &&&
[(a \bullet a \bullet a \bullet a) \circ j] \bullet i \ar[d]^-{[(\mu \bullet \mu)\circ 1] \bullet 1} \\
& i \ar[r]^-{\theta} \ar@{}@/_1.2pc/[rrr]_-{\omega} &
a\bullet i \ar[r]^-{\delta \bullet 1} &
(a\circ a ) \bullet  i \ar[r]^-{(1\circ \sigma) \bullet 1} &
(a\circ a^-) \bullet  i \ar[r]_-{\varphi \bullet 1} &
[(a \bullet a) \circ j] \bullet  i \bullet  i \ar[r]_-{1\bullet \xi_0} &
[(a \bullet a) \circ j] \bullet  i}$}
\caption{Proof of Lemma \ref{lem:Omega}~(3)}
\label{fig:Lem5}
\end{amssidewaysfigure}

\begin{proposition} \label{prop:Prop6}
Consider a naturally Frobenius map monoidale $M$ in a monoidal bicategory $\mathcal B$, and the duoidal category $\mathcal B(M,M)$ of Section \ref{sec:B(M,M)}.
\begin{itemize}
\item[{(1)}]
For a Hopf monoid 
$(a,(\mu,\eta),(\delta,\varepsilon),\sigma)$ 
in $\mathcal B(M,M)$, and a left integral $\theta$, 
the corresponding morphism $\Omega$ of Lemma \ref{lem:Omega}
renders commutative the following diagram.
$$
\xymatrix{
a \ar[r]^-{\Omega \bullet 1}  \ar[d]_-{1 \bullet \Omega} &
a \bullet a \bullet a \ar[d]^-{1 \bullet \mu} \\
a \bullet a \bullet a \ar[r]_-{\mu \bullet 1} &
a \bullet a}
$$
\item[{(2)}] If the integral $\theta$ of part (1) is normalized then the equal paths around the diagram of part (1) provide an $a$-bimodule section of the multiplication $\mu$. Thus the $\bullet$-monoid $(a,\mu,\eta)$ is separable.
\end{itemize}
\end{proposition}

\begin{proof}
(1) In both Figures \ref{fig:Prop6upper} and \ref{fig:Prop6lower}, $\omega$ denotes the morphism of \eqref{eq:omega}.
The  top-right paths of the commutative diagrams in Figure \ref{fig:Prop6upper} and in Figure \ref{fig:Prop6lower} coincide. Hence their left-bottom paths are equal. The region marked by $(\ast)$ in Figure \ref{fig:Prop6upper}  commutes since by the right $i$-linearity of $\theta$ also $\omega$ is right $i$-linear.

(2) The top-right path of the diagram of part (1) is obviously a right $a$-module morphism and the left-bottom path is a left $a$-module morphism. It follows by the associativity and the unitality of the $\bullet$-monoid $(a,\mu,\eta)$, together with Lemma \ref{lem:Omega}~(2), that either path around the diagram of part (1) provides a section of $\mu$.
\end{proof}

\begin{amssidewaysfigure}
\thisfloatpagestyle{empty}
{\color{white} .}
\hspace*{-2cm}
\vspace{2cm}
\centering
\scalebox{1}{$
\xymatrix@C=45pt@R=25pt{
a \ar@{=}[rr] \ar@{=}[d] \ar@/_7.5pc/[ddddd]_(.75){\Omega \bullet 1} &&
a \ar[dd]^-{\omega \circ 1} \\
(i\circ j) \bullet (i\circ a) \ar[r]^-\xi \ar[d]^-{(\omega \circ 1) \bullet 1} &
(i\bullet i) \circ a \ar[ru]_-{\xi_0 \circ 1} \ar[d]^-{(\omega \bullet 1) \circ 1} 
\ar@{}[rd]^-{(\ast)} \\
([(a\circ a^-) \bullet i] \circ j) \bullet a \ar[d]^-{[(\varphi \bullet 1) \circ 1] \bullet 1} &
[(a\circ a^-) \bullet i \bullet i] \circ a \ar[d]^-{(\varphi \bullet 1 \bullet 1) \circ 1} 
\ar[r]^-{(1\bullet \xi_0) \circ 1} &
[(a\circ a^-) \bullet i] \circ a \ar[d]^-{(\varphi \bullet 1 ) \circ 1} \\
[([(a\bullet a) \circ j]  \bullet i  \bullet i) \circ j] \bullet a \ar[d]^-{[(1\bullet \xi_0) \circ 1] \bullet 1} &
[([(a\bullet a) \circ j]  \bullet i  \bullet i \bullet i)] \circ a \ar[d]^-{(1\bullet \xi_0 \bullet 1) \circ 1}
\ar[r]^-{(1\bullet 1 \bullet \xi_0 ) \circ 1} &
[([(a\bullet a) \circ j] \bullet i \bullet i)] \circ a \ar[d]^-{(1\bullet \xi_0) \circ 1} \\
\qquad [\! ( [\! (a\! \bullet a\! )\! \circ\!  j\! ]\!   \bullet \! i) \! \circ \! j \! ] \! \bullet \! (\! i\! \circ\! a \!) 
\ar[r]^-\xi \ar[d]^-{\lambda \bullet 1} 
 \ar@{}[rrd]|-{\eqref{eq:Lem.5}} &
[([(a\bullet a) \circ j] \bullet i \bullet i)] \circ a \ar[r]^-{(1\bullet \xi_0) \circ 1} &
[([(a\bullet a) \circ j] \bullet i)] \circ a \ar[d]^-\lambda \\
a\bullet a \bullet a \ar@{=}[rr] &&
a\bullet a \bullet a \ar[d]^-{1\bullet \mu} \\
&& a\bullet a}$}
\caption{Proof of Proposition \ref{prop:Prop6} --- computation of the upper path}
\label{fig:Prop6upper}
\end{amssidewaysfigure}

\begin{amssidewaysfigure}
\thisfloatpagestyle{empty}
{\color{white} .}
\hspace*{-2cm}
\vspace{2cm}
\centering
\scalebox{1}{$
\xymatrix@C=10pt@R=25pt{
a \ar[r]^-{\raisebox{8pt}{${}_{\omega \circ 1}$}} \ar@{=}[d] &
[(a \circ a^-) \bullet i] \circ a \ar[r]^-{\raisebox{8pt}{${}_{(\varphi \bullet 1) \circ 1}$}} &
([(a\bullet a) \circ j] \bullet i \bullet i) \circ a \ar[r]^-{\raisebox{8pt}{${}_{(1\bullet \xi_0) \circ 1}$}} &
([(a\bullet a) \circ j] \bullet i) \circ a  \ar[rrrr]^-{\raisebox{8pt}{${}_\lambda$}} &&&&
a \bullet a \bullet a \ar[ddddddd]^-{1\bullet \mu} \\
a \ar[r]^-\delta  \ar@{=}[ddd]  \ar[rdd]^-\delta &
a\circ a \ar[rdd]^-{\delta \circ 1} \ar@{}[rrd]|-{
\textrm{Figure \ref{fig:Lem5}}} 
\ar[lu]_-{\varepsilon \circ 1} &&
([(a \bullet a \bullet a \bullet a) \circ j] \bullet i) \circ a \ar[u]_-{([(\mu \bullet \mu)\circ 1]\bullet 1) \circ 1} &&&
a \bullet a \bullet a \bullet a \bullet a \ar[ru]_-{\mu \bullet \mu \bullet 1}  
\ar[ddddd]^-{1\bullet 1\bullet 1\bullet  \mu} \\
&&& ([(a \bullet a) \circ j] \bullet i) \circ a \ar[rr]^-\lambda 
\ar[u]_-{([(1 \bullet \Omega \bullet 1) \circ 1]\bullet 1) \circ 1}  &
\ar@{}[ld]|-{\eqref{eq:Lem.7}} &
a \bullet a \bullet a \ar[ru]_-{1\bullet \Omega \bullet 1 \bullet 1} \ar[ddd]^-{1\bullet \mu} \\
& a\circ a \ar[r]^-{1\circ \delta} \ar[ld]^(.3){1\circ \varepsilon} &
a\circ a\circ a \ar[r]^-{1\circ \sigma \circ 1} 
\ar@{}[rd]|-{{\mbox{\tiny \cite[Theorem 7.2]{BohmLack}}}} &
a \circ a^- \circ a \ar[r]^-{1\circ \psi}  \ar[u]_-{\varphi \circ 1} &
a\circ (i \bullet [j\circ (a\bullet a)]) \ar[d]^-{1\circ [1\bullet (1\circ \mu)]} \ar[ru]_-\varrho \\
a \ar[rr]^-{1\circ (1\bullet \xi^0)} \ar@{=}[d] \ar@{}[rd]|(.4){\eqref{eq:Lem.6}} &&
a\circ [i \bullet (j\circ j)] \ar[lld]^-\varrho \ar[rr]^-{1\circ [1\bullet (1\circ \eta)]} &&
a\circ [i \bullet (j\circ a)] \ar[rd]^-\varrho \\
a \ar[rrrrr]^(.55){1\bullet \eta} \ar[d]_-{1\bullet \Omega} &&&&&
a \bullet a \ar[rd]^-{1\bullet \Omega \bullet 1} \\
a \bullet a \bullet a \ar[d]_-{\mu \bullet 1} \ar[rrrrrr]^-{1\bullet 1\bullet 1 \bullet \eta} &&&&&&
a \bullet a \bullet a \bullet a \ar[rd]^-{\mu \bullet \mu} \\
a \bullet a \ar@{=}[rrrrrrr] &&&&&&&
a \bullet a}$}
\caption{Proof of Proposition \ref{prop:Prop6} --- computation of the lower path}
\label{fig:Prop6lower}
\end{amssidewaysfigure}

For a naturally Frobenius map monoidale $M$ in a monoidal bicategory $\mathcal B$, and
the duoidal category $\mathcal B(M,M)$ of Section \ref{sec:B(M,M)}, we may regard right $i$-modules as Eilenberg-Moore algebras of the monad $- \bullet i$ on $\mathcal B(M,M)$ (whose multiplication is induced by $\xi_0$ and whose unit is induced by $\xi^0_0$). Let us denote their category by 
$\mathcal B(M,M)^{- \bullet i}$.
For a monoid $a$ in $(\mathcal B(M,M),\bullet)$, we may regard left $a$-modules - right $i$-modules as Eilenberg-Moore algebras of the 
monad $a\bullet - \bullet i$ on $\mathcal B(M,M)$ --- equivalently, as Eilenberg-Moore algebras of the 
monad $a\bullet - $ on $\mathcal B(M,M)^{- \bullet i}$. We denote their category by $\mathcal B(M,M)^{a\bullet - \bullet i}$; and we use similar notations when the roles of left and right actions are interchanged.

The following Maschke type theorem is our first main result.

\begin{theorem} \label{thm:Maschke}
Consider a naturally Frobenius map monoidale $M$ in a monoidal bicategory $\mathcal B$. For a Hopf monoid 
$(a,(\mu,\eta),(\delta,\varepsilon),\sigma)$ in the duoidal category $\mathcal B(M,M)$ of Section \ref{sec:B(M,M)}, the following assertions are equivalent.
\begin{itemize}
\item[{(i)}] 
The monoid  $(a,\mu,\eta)$ in $(\mathcal B(M,M),\bullet,j)$ is separable.
\item[{(ii)}] There exists a normalized left integral $i \to a\bullet i$.
\item[{(iii)}] There exists a normalized right integral $i \to i\bullet a$.
\item[{(iv)}] The forgetful functor $\mathcal B(M,M)^{a\bullet - \bullet i}\to \mathcal B(M,M)^{- \bullet i}$ is separable (in the sense of \cite[page 398]{NVdBVO}).
\item[{(v)}] 
The forgetful functor $\mathcal B(M,M)^{i\bullet -\bullet a}\to \mathcal B(M,M)^{i\bullet -}$ is separable.
\item[{(vi)}] 
The forgetful functor $\mathcal B(M,M)^{a\bullet - \bullet i}\to \mathcal B(M,M)^{- \bullet i}$ reflects split epimorphisms.
\item[{(vii)}] 
The forgetful functor $\mathcal B(M,M)^{i\bullet -\bullet a}\to \mathcal B(M,M)^{i\bullet -}$ reflects split epimorphisms.
\end{itemize}
\end{theorem}

\begin{proof}
It is well-known that (i)$\Rightarrow$(iv), (v) --- see e.g. \cite[Paragraph 2.9]{BBW} --- and that (iv)$\Rightarrow$(vi) and (v)$\Rightarrow$(vii) --- see \cite[Proposition 1.2]{NVdBVO}.
By the bimonoid axiom $\varepsilon.\eta=\xi^0_0$, the right vertical of
$$
\xymatrix{
i \ar@{=}[r] \ar@{=}@/_2pc/[rrdd] &
j \bullet i \ar[r]^-{\eta \bullet 1} \ar[rd]_-{\xi^0_0 \bullet 1} &
a \bullet i \ar[d]^-{\varepsilon \bullet 1} \\
&& i \bullet i \ar[d]^-{\xi_0} \\
&& i}
$$
is an epimorphism of right $i$-modules split by the top row. It is a morphism of left $a$-modules by the commutativity of 
$$
\xymatrix@C=40pt{
a\bullet a \bullet i \ar[r]^-{1 \bullet \varepsilon \bullet 1} \ar[dd]_-{\mu \bullet 1} &
a\bullet i \bullet i \ar[r]^-{1\bullet \xi_0} \ar[d]^-{\varepsilon \bullet 1 \bullet 1} &
a\bullet i \ar[d]^-{\varepsilon \bullet 1} \\
& i \bullet i \bullet i \ar[r]^-{1\bullet \xi_0} \ar[d]^-{\xi_0 \bullet 1} &
i \bullet i  \ar[d]^-{\xi_0} \\
a\bullet i \ar[r]_-{\varepsilon \bullet 1} &
i \bullet i  \ar[r]_-{\xi_0} &
i}
$$
whose rectangle on the left commutes by a bimonoid axiom.
Thus if (vi) holds then it is a split epimorphism of left $a$-modules - right $i$-modules as well whence its section is a normalized left integral. This proves (vi)$\Rightarrow$(ii).
The implication (vii)$\Rightarrow$(iii) follows symmetrically.
We infer from Proposition \ref{prop:Prop6} that (ii)$\Rightarrow$(i) and (iii)$\Rightarrow$(i) follows symmetrically.
\end{proof}

\begin{remark} \label{rem:TuraevVirelizier}
Any bimonoid $(a,(\mu,\eta),(\delta,\varepsilon))$ in a duoidal category $(\mathsf D,\bullet,\circ)$ induces an opmonoidal monad $a\bullet (-)$ on the monoidal category $(\mathsf D,\circ)$, see \cite[Theorem 6.7]{BookerStreet}. Application of the Maschke type theorem \cite[Theorem 8.11]{TuraevVirelizier} to it yields the equivalence of the following assertions.
\begin{itemize}
\item[{($\ast$)}] The counit of the adjunction $a\bullet - \dashv \textrm{ forgetful functor}:\mathsf D^{a\bullet -}\to \mathsf D$ possesses a right inverse. 
\item[{(\raisebox{-3pt}{$\stackrel {\displaystyle\ast}\ast$})}] For any object $(v,\xymatrix@C=12pt{a\bullet v \ar[r]^-\nu & v})$ of $\mathsf D^{a\bullet -}$, $\nu$ is a split epimorphism in $\mathsf D^{a\bullet -}$.
\item[{(\raisebox{-3pt}{$\stackrel{\displaystyle \ast\ast}{\, \ast}$})}] $\xymatrix@C=15pt{a\bullet i \ar[r]^-{\varepsilon \bullet 1}& i\bullet i \ar[r]^-{\xi^0} & i}$ 
is a split epimorphism in $\mathsf D^{a\bullet -}$.
\end{itemize}

Consider now a naturally Frobenius map monoidale $M$ in a monoidal bicategory $\mathcal B$. 
There are several easy ways to see that the above equivalent conditions hold for a Hopf monoid  
in the duoidal category $\mathcal B(M,M)$ of Section \ref{sec:B(M,M)} if it satisfies the equivalent conditions of Theorem \ref{thm:Maschke} (while there is no reason to expect the converse in general). For example, Theorem \ref{thm:Maschke}~(ii)$\Rightarrow\ (\raisebox{-3pt}{$\stackrel{\displaystyle \ast\ast}{\, \ast}$})$ is trivial.
\end{remark}


\section{Coseparability of a Hopf monoid} \label{sec:DualMaschke}

Recall that a comonoid in a monoidal category  is said to be {\em coseparable} if it is a separable monoid in the opposite category. That is, its comultiplication admits a bicomodule retraction.
The aim of this section is to find sufficient and necessary conditions for the coseparability of the constituent comonoid of a Hopf monoid in the duoidal category $\mathcal B(M,M)$ 
for a naturally Frobenius map monoidale $M$ in some monoidal bicategory $\mathcal B$.

Both monoidal structures of the duoidal category $\mathcal B(M,M)$ come  from different sources: one of them is a composition while the other one is a convolution in $\mathcal B$. Their roles can not be interchanged. Therefore the results of this section can not be obtained from those in Section \ref{sec:Maschke} by some kind of duality; they need independent proofs.

\begin{definition} \label{def:cointegral}
By a {\em left cointegral} for a bimonoid $(a,(\mu,\eta),(\delta,\varepsilon))$ in a duoidal category $(\mathsf D,\bullet,\circ)$ we mean a left integral for the bimonoid $(a,(\delta,\varepsilon),(\mu,\eta))$ in the duoidal category  $(\mathsf D\op,\circ,\bullet)$.
That is, a left $a$-comodule right $j$-comodule morphism $\theta:a\circ j \to j$; meaning that the first two diagrams of 
\begin{equation} \label{eq:cointegral}
\xymatrix{
a \circ j \ar[r]^-\theta \ar[dd]_-{\delta \circ 1} &
j \ar[d]^-{\xi^0} \\
& j\circ j \ar[d]^-{\eta \circ 1} \\
a\circ a \circ j \ar[r]_-{1 \circ \theta} &
a\circ j}
\qquad \qquad
\xymatrix@R=66pt{
a \circ j \ar[r]^-\theta \ar[d]_-{1\circ \xi^0} &
j \ar[d]^-{\xi^0} \\
a  \circ j \circ j \ar[r]_-{\theta \circ 1} &
j\circ j}
\qquad \qquad
\xymatrix{
j \ar[d]_-{\xi^0} \ar@{=}[rdd] \\
j\circ j \ar[d]_-{\eta \circ 1} \\
a \circ j \ar[r]_-\theta &
j}
\end{equation}
commute. A left cointegral $\theta$ is said to be {\em normalized} if it is normalized as a left integral; that is, also the third diagram above commutes.

Dually, a {\em right cointegral} for a bimonoid $a$ in $(\mathsf D,\bullet,\circ)$ is a left cointegral for $a$ regarded as a bimonoid in $(\mathsf D,\bullet\rev,\circ\rev)$. That is, a right $a$-comodule left  $j$-comodule morphism $\theta:j\circ a \to j$. A right cointegral is {\em normalized} if it is normalized as a left cointegral; that is, it is a retraction of the $a$-coaction
$\xymatrix@C=15pt{
j \ar[r]^-{\xi^0} & j \circ j \ar[r]^-{1\circ \eta} & j \circ a}$.
\end{definition}

Consider a naturally Frobenius map monoidale $M$ in a monoidal bicategory $\mathcal B$ and the duoidal category $\mathcal B(M,M)$ of Section \ref{sec:B(M,M)}.
For a Hopf monoid 
$(a,(\mu,\eta),(\delta,\varepsilon),\sigma)$ in $\mathcal B(M,M)$, 
and any morphism $\theta: j\circ a \to j$,
consider the composite morphism
\begin{equation} \label{eq:Theta}
\Theta:=(
\xymatrix{
a \circ a \ar[r]^-{\sigma \circ 1} &
a^- \circ a \ar[r]^-\psi &
i\bullet [j\circ (a\bullet a)] \ar[r]^-{1\bullet (1\circ \mu)} &
i \bullet (j\circ a) \ar[r]^-{1\bullet \theta} &
i}).
\end{equation}

\begin{lemma} \label{lem:Theta}
Consider a naturally Frobenius map monoidale $M$ in a monoidal bicategory $\mathcal B$ 
and a Hopf monoid 
$(a,(\mu,\eta),(\delta,\varepsilon),\sigma)$ 
in the duoidal category $\mathcal B(M,M)$ of Section \ref{sec:B(M,M)}.
For a morphism $\theta: j\circ a \to j$
such that the rightmost diagram of \eqref{eq:cointegral} commutes,
the associated morphism $\Theta$ of \eqref{eq:Theta} satisfies $\Theta.\delta=\varepsilon$.
\end{lemma}

\begin{proof}
The claim follows by the commutativity of the next diagram, whose top-right path is equal to $\Theta$.
$$
\xymatrix{
a\circ a \ar[r]^-{\sigma \circ 1} \ar@{}[rrrd]|-{\mbox{\tiny \cite[Theorem 7.2]{BohmLack}}} &
a^- \circ a \ar[r]^-\psi & 
i\bullet [j\circ (a\bullet a)] \ar[r]^-{1\bullet (1\circ \mu)} &
i \bullet (j\circ a) \ar@/^2pc/[rd]^-{1\bullet \theta} \\
a \ar[u]^-\delta \ar[r]_-\varepsilon &
i \ar[rr]^-{1\bullet \xi^0} \ar@/_1.4pc/@{=}[rrr] &&
i\bullet (j\circ j) \ar[u]^-{1\bullet (1 \circ \eta)} &
i\\
&&&}
\vspace*{-1cm}
$$
\end{proof}

\begin{proposition} \label{prop:Theta}
Consider a naturally Frobenius map monoidale $M$ in a monoidal bicategory $\mathcal B$, and the duoidal category $\mathcal B(M,M)$ of Section \ref{sec:B(M,M)}.
\begin{itemize}
\item[{(1)}] For a Hopf monoid 
$(a,(\mu,\eta),(\delta,\varepsilon),\sigma)$ 
in $\mathcal B(M,M)$, and a right cointegral $\theta$,
the associated morphism $\Theta$ of \eqref{eq:Theta} renders commutative the following diagram.
$$
\xymatrix{
a\circ a \ar[r]^-{\delta \circ 1} \ar[d]_-{1\circ \delta} &
a \circ a\circ a \ar[d]^-{1\circ \Theta} \\
a \circ a\circ a \ar[r]_-{\Theta \circ 1} & 
a}
$$
\item[{(2)}] If the right cointegral $\theta$ of part (1) is normalized, then the equal paths around the diagram of part (1) provide an $a$-bicomodule retraction of the comultiplication $\delta$.
\end{itemize}
\end{proposition}

\begin{proof}
(1) The region of Figure \ref{fig:Theta_upper} marked by the symbol $(\ast)$ commutes by a bimonoid axiom. The region marked by  \cite[Lemma 4.3~(iii)]{BohmLack} commutes by the application of \cite[Lemma 4.3~(iii)]{BohmLack} to the naturally Frobenius map monoidale $(M,m^*,u^*)$ in $\mathcal B^{\mathsf{op},\mathsf{rev}}$, see Remark \ref{rem:rev}.
The bottom rows of the commutative diagrams in Figure \ref{fig:Theta_upper} and in Figure \ref{fig:Theta_lower} coincide. The left columns are equal by \eqref{eq:step3}. Therefore also the top paths are equal. 

(2) The top-right path of the diagram of part (1) is obviously a morphism of left $a$-comodules and the left-bottom path is a morphism of right $a$-comodules. Either path is a retraction of $\delta$ by the coassociativity and the counitality of the comonoid $(a,\delta,\varepsilon)$, applied together with Lemma \ref{lem:Theta}.
\end{proof}

\begin{amssidewaysfigure}
\thisfloatpagestyle{empty}
{\color{white} .}
\hspace*{-2.4cm}
\vspace*{2cm}
\centering
\scalebox{1}{$
\xymatrix@C=11pt@R=30pt{
\\
\\
\\
a\circ a \ar[d]^-{\delta \circ 1} \\
a\circ a\circ a \ar[d]^-{1\circ \sigma \circ 1} \ar[rrrrrr]^-{1\circ \Theta} &&&&&&
a \ar@{=}[d] \ar@{=}@/^3pc/[ddd] \\
a\circ a^- \circ a \ar[d]^-{1\circ \delta^- \circ 1} \ar[rr]^-{1\circ \psi} &&
a\circ (i\bullet [j\circ (a\bullet a)]) \ar[rrr]^-{1\circ [1\bullet (1\circ \mu)]} 
\ar[d]^-{1\circ (1\bullet [1\circ (\delta \bullet \delta)])} 
\ar@{}[rrrdd]|-{(\ast)} &&&
a\!\circ \![\!i\!\bullet\! (\!j\!\circ \!a\!)\!] 
\ar[r]^-{\raisebox{8pt}{${}_{ 1\circ (1\bullet \theta)}$}}
\ar[dd]_-{1\circ [1\bullet (1\circ \delta)]}
\ar@{}[rdd]|(.4){\eqref{eq:cointegral}} &
a\! \circ \! (\! i\! \bullet\! j\! ) \ar[d]_-{1\circ (1\bullet \xi^0)} \ar@{}[rd]|(.5){\eqref{eq:Lem.6}}  \\
a\circ (a\circ a)^- \circ a \ar[ddd]^-\cong \ar[r]^-{1\circ 1\circ \delta} &
a\circ (a\circ a)^- \circ a \circ a \ar[ddd]_-\cong \ar[r]^-{1\circ \psi} 
\ar@{}[rddd]|-{\mbox{\tiny \cite[Lemma 4.3~(iii)]{BohmLack}}} &
a\! \circ \![\! i\! \bullet\! (\! j\! \circ\! [\! (\! a \! \circ \! a\! ) \! \bullet \! (\! a \! \circ \! a\! )\! ] )\! ]
\ar[d]^-{1\circ [1\bullet (1\circ \xi)]} &&&&
a\! \circ\! [\! i\! \bullet \! (\! j\! \circ \! j\! )\! ] \ar[d]_-\varrho  
\ar[ldd]_(.3)*-<.5em>{_{1\circ [1\bullet (1\circ \eta)]}}   & \\
&& a\!\circ\! (\!i\!\bullet\! [\! j\! \circ \! (\! a\! \bullet\! a\!)\! \circ \!(\! a\! \bullet \! a\! )\!]) 
\ar[r]^-{\raisebox{8pt}{${}_{1\circ [1\bullet (1\circ \mu \circ 1)]}$}} &
a\!\circ\! (\!i\!\bullet\! [\! j\! \circ \! a\! \circ \!(\! a\! \bullet \! a\! )\!]) 
\ar[rr]^-{1\circ [1\bullet (1\circ 1 \circ \mu)]}
\ar[rrdd]^(.4){1\circ [1\bullet (\theta \circ 1)]} 
\ar@{}[rdd]|-{\eqref{eq:step4}} &&
a\!\circ\! [\!i\! \bullet \!(\!j\!\circ \!a\! \circ\! a\!)\!] 
\ar[d]_-{1\circ [1\bullet (\theta \circ 1)]} &
a \ar[d]_-{1\bullet \eta} \ar@{=}@/^3pc/[dddd] \\ 
&&&&& a\!\circ\! [\!i\bullet\! (\!j\! \circ \!a\!)\!] \ar[r]^-\varrho &
a\bullet a \ar@{=}[dd] \\
a\circ a^- \circ a^- \circ a \ar[r]^-{1\circ 1\circ 1 \circ \delta} \ar[d]^-{\varphi \circ 1 \circ 1} &
a\circ a^- \circ a^- \circ a \circ a \ar[r]^-{1\circ 1 \circ \psi \circ 1} &
a\!\circ\! a^- \!\circ \!(\!i\!\bullet\! [\!j\!\circ\! (\!a\!\bullet\! a\!)\!])\! \circ \!a 
\ar[r]^-{\raisebox{8pt}{${}_{1\circ 1 \circ [1\bullet (1\circ \mu)] \circ 1}$}}
\ar[uu]_-{1\circ \kappa} &
a\!\circ\! a^-\! \circ \![\!i\!\bullet \!(\!j\!\circ \!a\!)\!]\! \circ\! a 
\ar[r]^-{\raisebox{8pt}{${}_{1\circ 1 \circ (1\bullet  \theta) \circ 1}$}} 
\ar[uu]_-{1 \circ \kappa} &
a\circ a^- \circ a \ar[r]^-{1\circ \psi} \ar[d]_-{\varphi \circ 1} \ar@{}[rd]|-{\eqref{eq:Lem.7}} &
a\!\circ\! (\!i\! \bullet\! [\!j\!\circ \!(\!a\!\bullet\! a\!)\!]) \ar[d]_-\varrho 
\ar[u]^(.7){1\circ [1\bullet (1\circ \mu)]} \\
\quad ([\!(\! a\!\bullet\! a\!) \!\circ\! j\!] \!\bullet \!i\!)\! \circ \!a^- \!\! \circ\! a \! 
\ar[d]^-{[(\mu \circ 1) \bullet 1] \circ 1 \circ 1} &&&&
([\!(\!a\!\bullet\! a) \!\circ \!j\!]\! \bullet\! i\!)\! \circ \! a \ar[r]^-\lambda \ar[d]_-{[(\mu \circ 1) \bullet 1] \circ 1} &
a\bullet a \bullet a \ar[r]^-{1\bullet \mu} \ar[d]_-{\mu \bullet 1} &
a\bullet a \ar[d]_-\mu \\
[\!(\!a\!\circ\! j\!)\!\bullet\! i\!]\! \circ\! a^- \!\circ\! a 
\ar[r]_-{\raisebox{-8pt}{${}_{1\circ 1 \circ \delta}$}}  &
[\!(\!a\!\circ\! j\!)\!\bullet\! i\!]\! \circ \! a^- \!\circ \!a \!\circ\! a 
\ar[r]_-{\raisebox{-8pt}{${}_{1 \circ \psi \circ 1}$}} &
[\!(\!a\!\circ\! j\!)\!\bullet\! i\!]\! \circ\! (\!i\!\bullet \![\!j\!\circ \!(\!a\!\bullet \!a)\!])\! \circ\! a 
\ar[r]_-{\raisebox{-8pt}{${}_{1 \circ [1\bullet (1\circ \mu)] \circ 1}$}} &
[\!(\!a\!\circ \!j\!)\!\bullet\! i\!]\! \circ\! [\!i\!\bullet \!(\!j\!\circ\! a\!)\!]\! \circ \!a 
\ar[r]_-{\raisebox{-8pt}{${}_{1 \circ (1\bullet \theta) \circ 1}$}} &
[\!(\!a\!\circ \!j)\!\bullet \!i]\!\circ\! a 
\ar[r]_-\lambda &
a\bullet a \ar[r]_-\mu & 
a}$}
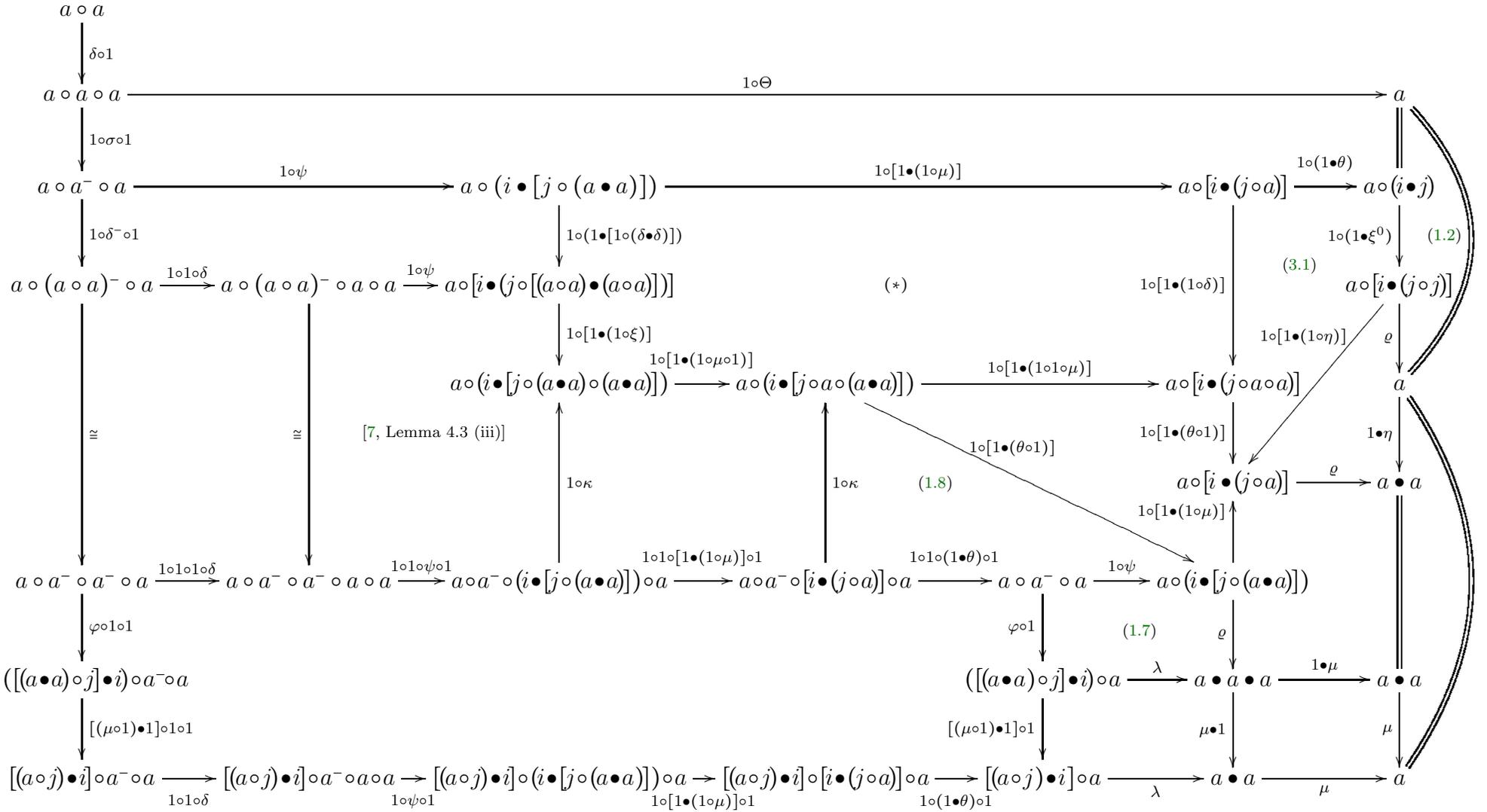
\caption{Proof of Proposition \ref{prop:Theta}~(1) --- computation of the upper path}
\label{fig:Theta_upper}
\end{amssidewaysfigure}

\begin{amssidewaysfigure}
\thisfloatpagestyle{empty}
{\color{white} .}
\hspace*{-2.2cm}
\vspace*{2cm}
\centering
\scalebox{1}{$
\xymatrix@C=17pt@R=30pt{
a \circ a \ar[r]^-{1\circ \delta} \ar[d]^-{\sigma \circ 1} &
a \circ a \circ a \ar[r]_-{\sigma \circ 1 \circ 1} \ar@/^3pc/[rrrrr]^-{\Theta \circ 1}&
a^- \circ a \circ a \ar[rr]_-{\psi \circ 1} \ar[dd]^-{(\xi^0 \bullet 1) \circ 1 \circ 1 \circ 1 } &&
(\! i \!\bullet\! [\!j\!\circ \!(\!a\! \bullet \! a\!)\!]) \circ a 
\ar[r]_-{\raisebox{-8pt}{${}_{[1\bullet (1\circ \mu)] \circ 1}$}} 
\ar@{}[rrdd]|(.7){\eqref{eq:Lem.6}} &
[\! i\! \bullet\! (\!j\!\circ\! a\!)\!]\! \circ\! a 
\ar[r]
&
a
\ar[lldd]_-{(\xi^0 \bullet 1) \circ 1}|(.1){(1\bullet \theta) \circ 1\quad }
\ar@{=}[dd] \\
a^- \circ a \ar[dd]^-{(\xi^0 \bullet 1) \circ 1 \circ 1} \\
&& [\!(\!j\!\circ\! j\!)\! \bullet\! i\!]\! \circ\! a^-\! \circ\! a\! \circ\! a 
\ar[ldd]^-{[(\eta \circ 1)\bullet 1] \circ 1 \circ 1 \circ 1} &&
[\!(\!j\!\circ\! j\!)\! \bullet\! i\!]\! \circ\! a \ar[rr]^-\lambda \ar[dd]_-{[(\eta \circ 1)\bullet 1] \circ 1} &&
a \ar[ldd]_-{\eta \bullet 1} \ar@{=}[dd] \\
[\!(\!j\!\circ\! j\!)\! \bullet\! i\!]\! \circ\! a^- \!\!\circ\! a \ar[d]^-{[(\eta \circ 1)\bullet 1] \circ 1 \circ 1 } \\
[\!(\!a\!\circ\! j\!)\!\bullet\! i\!]\! \circ\! a^- \!\circ\! a 
\ar[r]_-{\raisebox{-8pt}{${}_{1\circ 1 \circ \delta}$}}  &
[\!(\!a\!\circ\! j\!)\!\bullet\! i\!]\! \circ \! a^- \!\circ \!a \!\circ\! a 
\ar[r]_-{\raisebox{-8pt}{${}_{1 \circ \psi \circ 1}$}} &
[\!(\!a\!\circ\! j\!)\!\bullet\! i\!]\! \circ\! (\!i\!\bullet \![\!j\!\circ \!(\!a\!\bullet \!a)\!])\! \circ\! a 
\ar[r]_-{\raisebox{-8pt}{${}_{1 \circ [1\bullet (1\circ \mu)] \circ 1}$}} &
[\!(\!a\!\circ \!j\!)\!\bullet\! i\!]\! \circ\! [\!i\!\bullet \!(\!j\!\circ\! a\!)\!]\! \circ \!a 
\ar[r]_-{\raisebox{-8pt}{${}_{1 \circ (1\bullet \theta) \circ 1}$}} &
[\!(\!a\!\circ \!j)\!\bullet \!i]\!\circ\! a 
\ar[r]_-\lambda &
a\bullet a \ar[r]_-\mu & 
a}$}
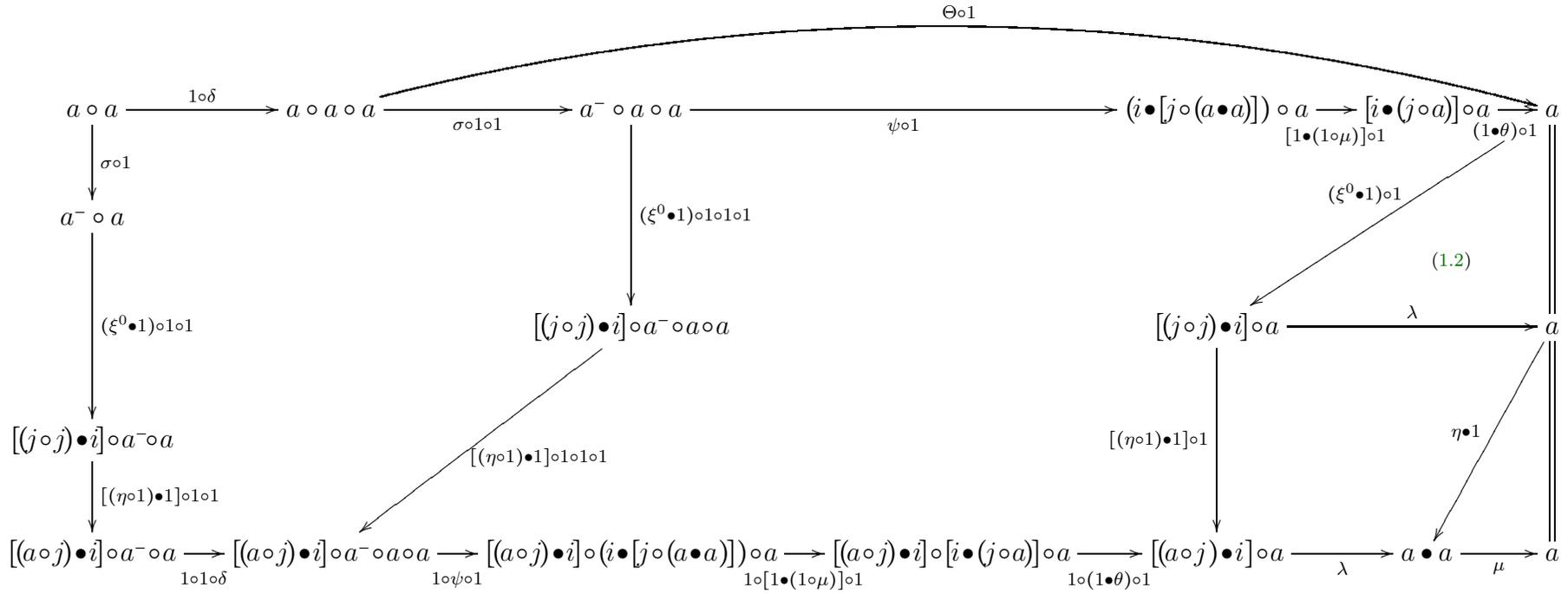
\caption{Proof of Proposition \ref{prop:Theta}~(1) --- computation of the lower path}
\label{fig:Theta_lower}
\end{amssidewaysfigure}

For a naturally Frobenius map monoidale $M$ in a monoidal bicategory $\mathcal B$, and
the duoidal category $\mathcal B(M,M)$ of Section \ref{sec:B(M,M)}, we may regard left $j$-comodules as Eilenberg-Moore coalgebras of the comonad $j\circ -$ on $\mathcal B(M,M)$ (whose comultiplication is induced by $\xi^0$ and whose counit is induced by $\xi^0_0$). Let us denote their category by 
$\mathcal B(M,M)^{j\circ -}$.
For a comonoid $a$ in $(\mathcal B(M,M),\circ)$, we may regard right $a$-comodules - left $j$-comodules as Eilenberg-Moore coalgebras of the comonad $j\circ - \circ a$ on $\mathcal B(M,M)$ --- equivalently, as Eilenberg-Moore coalgebras of the comonad $- \circ a$ on $\mathcal B(M,M)^{j\circ -}$. We denote their category by $\mathcal B(M,M)^{j\circ - \circ a}$; and we use similar notations when the roles of left and right coactions are interchanged.

The Maschke type theorem below is our second main result.

\begin{theorem} \label{thm:DualMaschke}
Consider a naturally Frobenius map monoidale $M$ in a monoidal bicategory $\mathcal B$. For a Hopf monoid 
$(a,(\mu,\eta),(\delta,\varepsilon),\sigma)$ in the duoidal category $\mathcal B(M,M)$ of Section \ref{sec:B(M,M)}, the following assertions are equivalent.
\begin{itemize}
\item[{(i)}] 
$(a,\delta,\varepsilon)$ is a coseparable comonoid in $(\mathcal B(M,M),\circ,i)$.
\item[{(ii)}] There exists a normalized left cointegral $a\circ j \to j$.
\item[{(iii)}] There exists a normalized right cointegral $j\circ a \to j$.
\item[{(iv)}] The forgetful functor $\mathcal B(M,M)^{a\circ -\circ j}\to \mathcal B(M,M)^{ -\circ j}$ is separable (in the sense of \cite[page 398]{NVdBVO}).
\item[{(v)}] The forgetful functor $\mathcal B(M,M)^{j\circ -\circ a}\to \mathcal B(M,M)^{j\circ -}$ is separable.
\item[{(vi)}] The forgetful functor $\mathcal B(M,M)^{a\circ -\circ j}\to \mathcal B(M,M)^{ -\circ j}$ reflects split monomorphisms.
\item[{(vii)}] The forgetful functor $\mathcal B(M,M)^{j\circ -\circ a}\to \mathcal B(M,M)^{j\circ -}$ reflects split monomorphisms.
\end{itemize}
\end{theorem}

\begin{proof}
It is well-known that (i)$\Rightarrow$(iv),(v) --- see e.g. \cite[Paragraph 2.9~(2)]{BBW} --- and that (iv)$\Rightarrow$(vi) and (v)$\Rightarrow$(vii) --- see \cite[Proposition 1.2~(1')]{NVdBVO}.
By the bimonoid axiom $\varepsilon.\eta=\xi^0_0$, the top row of
$$
\xymatrix{
j\ar[r]^-{\xi^0} \ar@{=}@/_2pc/[rrdd] &
j \circ j \ar[r]^-{1 \circ \eta} \ar[rd]_-{1 \circ \xi^0_0} &
j\circ a \ar[d]^-{1 \circ \varepsilon} \\
&& j\circ i \ar@{=}[d] \\
&& j}
$$
is a monomorphism of left $j$-comodules split by the right column. 
It is a morphism of right $a$-comodules by the commutativity of 
$$
\xymatrix@C=40pt{
j \ar[r]^-{\xi^0} \ar[d]_-{\xi^0} &
j \circ j \ar[r]^-{1\circ \eta} \ar[d]^-{1\circ \xi^0} &
j \circ a \ar[dd]^-{1\circ \delta} \\
j \circ j \ar[r]^-{\xi^0 \circ 1} \ar[d]_-{1\circ \eta} &
j \circ j \circ j  \ar[d]^-{1\circ 1\circ  \eta} \\
j \circ a \ar[r]_-{\xi^0 \circ 1} &
j \circ j \circ a \ar[r]_-{1\circ \eta \circ 1} &
j \circ a \circ a}
$$
whose rectangle on the right commutes by a bimonoid axiom.
Thus if (vii) holds then it is a split monomorphism of right $a$-comodules - left $j$-comodules as well whence its retraction is a normalized right cointegral. This proves (vii)$\Rightarrow$(iii).
The implication (vi)$\Rightarrow$(ii) follows symmetrically.
We infer from Proposition \ref{prop:Theta} that (iii)$\Rightarrow$(i) and (ii)$\Rightarrow$(i) follows symmetrically.
\end{proof}

\section{Applications} \label{sec:applications}

\subsection{Hopf monoids in braided monoidal categories}
Any monoidal category $(\mathsf C,\otimes,k)$ can be regarded as a bicategory with a single object, the 1-cells provided by the objects of $\mathcal C$ and the 2-cells provided by the morphisms of $\mathsf C$. The vertical composition is the composition in $\mathsf C$ while the horizontal composition is the monoidal product $\otimes$.

If furthermore the monoidal category $(\mathsf C,\otimes,k)$ is braided, then the reverse of $\otimes$ renders the above bicategory monoidal.
Its single object is a trivial naturally Frobenius map monoidale, hence its endohom category $\mathsf C$ possesses a duoidal structure as in Section \ref{sec:B(M,M)}. Both monoidal structures turn out to be $(\otimes,k)$ with compatibility morphism $\xi$ determined by the braiding; and the other compatibility morphisms $\xi^0$, $\xi_0$ and $\xi^0_0$ given by the unitality natural isomorphisms, as in \cite[Section 6.3]{AguiarMahajan}. 

As the bimonoids (respectively, Hopf monoids) in this duoidal category we re-obtain the usual notion of bimonoids (respectively, Hopf monoids) in the braided monoidal category $(\mathsf C,\otimes,k)$; see e.g. \cite[pages 113-114]{majid}.

A left (or right) integral in the sense of Definition \ref{def:integral} for a bimonoid $(a,(\mu,\eta),(\delta,\varepsilon))$ in a braided monoidal category $(\mathsf C,\otimes,k)$ --- regarded as a duoidal category --- is just a left (or right) $a$-module morphism $k \to a$. It is normalized in the sense of Definition \ref{def:integral} if it is a section of the counit $\varepsilon$.

Applying Theorem  \ref{thm:Maschke} to the above situation we obtain the following.

\begin{theorem}
For a Hopf monoid $(a,(\mu,\eta),(\delta,\varepsilon),\sigma)$ in a braided monoidal category $(\mathsf C,\otimes,k)$, the following assertions are equivalent.
\begin{itemize}
\item[{(i)}] $(a,\mu,\eta)$ is a separable monoid in $(\mathsf C,\otimes,k)$.
\item[{(ii)}] $a$ admits a normalized left integral, that is, a left $a$-module section of $\varepsilon$.
\item[{(iii)}] $a$ admits a normalized right integral, that is, a right $a$-module section of $\varepsilon$.
\end{itemize}
\end{theorem}

Application of Theorem  \ref{thm:DualMaschke} to a Hopf monoid $(a,(\mu,\eta),(\delta,\varepsilon),\sigma)$ in a braided monoidal category $(\mathsf C,\otimes,k)$ yields nothing new in this case but the same as the application of 
Theorem  \ref{thm:Maschke} to $(a,(\mu,\eta),(\delta,\varepsilon),\sigma)$ regarded as a Hopf monoid in the opposite of the category $\mathsf C$.

\subsection{Weak Hopf algebras}
A {\em weak bialgebra} in \cite[Definition 2.1]{BNSz}, over a field $k$, is a vector space $A$ equipped with an algebra structure $(\mu,\nu)$ and a coalgebra structure $(\Delta,\epsilon)$ subject to the axioms formulated as the commutativity of the diagrams below.
$$
\xymatrix@C=45pt{
A\otimes A \ar[d]_-\mu \ar[r]^-{\Delta \otimes \Delta} &
A\! \otimes \! A \!  \otimes \!  A \!  \otimes \!  A \ar[r]^-{1\otimes \mathsf{flip} \otimes 1} &
A\! \otimes \! A \!  \otimes \!  A \!  \otimes \!  A \ar[d]^--{{ \mu} \otimes {\mu}}   \\
A \ar[rr]_-\Delta &&
A \otimes A}
$$
where $\mathsf{flip}:A\otimes A \to A\otimes A$ is the flip map $a\otimes b \mapsto b\otimes a$, and
$$
\xymatrix{
k\ar[rr]^-{\nu  \otimes \nu}
\ar[d]^-{\nu \otimes \nu} 
\ar@/^.5pc/[rd]^-{ \nu} &&
A\otimes A\ar[d]_-{\Delta \otimes \Delta}
\\
A\otimes A\ar[d]^-{ \Delta \otimes \Delta}&
A\ar@/_.5pc/[rd]_-{\Delta^2}&
A\! \otimes \! A \!  \otimes \!  A \!  \otimes \!  A \ar[d]_(.3){ 1{\scalebox{.5}{$\otimes$}}
\mu\scalebox{.5}{$\otimes$}  1}
\\
A\! \otimes \! A \!  \otimes \!  A \!  \otimes \!  A\ar[rr]_-{ 1\otimes \mu^{\mathsf{op}}\otimes 1}&&
A\otimes A \otimes A}
\quad 
\xymatrix{
A\otimes A \otimes A \ar[rr]^-{ 1\otimes \Delta^{\mathsf{op}}\otimes 1}
\ar[d]^(.6){ 1\scalebox{.5}{$\otimes$}\Delta\scalebox{.5}{$\otimes$}1}
\ar@/^.5pc/[rd]^-{ \mu^2}&&
A\! \otimes \! A \!  \otimes \!  A \!  \otimes \!  A\ar[d]_-{{ \mu} \otimes { \mu}}\\
A\! \otimes \! A \!  \otimes \!  A \!  \otimes \!  A\ar[d]^-{{ \mu} \otimes {\mu}}&
A\ar@/_.5pc/[rd]_-{ \epsilon}&
A\otimes A \ar[d]_-{{ \epsilon} \otimes { \epsilon}}\\
A\otimes A \ar[rr]_-{{ \epsilon} \otimes { \epsilon}}&&
k}
$$
where $\mu^2:\mu.(\mu \otimes 1)=\mu.(1 \otimes \mu)$ and
$\Delta^2:=(\Delta \otimes 1).\Delta=(1 \otimes \Delta).\Delta$,
$\mu\op:=\mu.\mathsf{flip}$ and $\Delta\op:=\mathsf{flip}.\Delta$.

While the first diagram requires the multiplicativity of the comultiplication in the usual sense of a proper bialgebra, unitality of the comultiplication; that is, $\Delta.\nu=\nu\otimes \nu$ is {\em not} required.

Essential roles are played by the following --- in fact idempotent --- linear maps.
$$
\begin{array}{rcl}
\sqcap^R &=& (\xymatrix{
A \ar[r]^-{\nu \otimes 1} &
A \otimes A \ar[r]^-{\Delta \otimes 1} &
A\otimes A \otimes A \ar[r]^-{1\otimes \mu^{\mathsf{op}}} &
A\otimes A \ar[r]^-{1\otimes \epsilon} &
A}) \\
\overline \sqcap^R &=& (\xymatrix{
A \ar[r]^-{\nu \otimes 1} &
A \otimes A \ar[r]^-{\Delta \otimes 1} &
A\otimes A \otimes A \ar[r]^-{1\otimes \mu} &
A\otimes A \ar[r]^-{1\otimes \epsilon} &
A}) \\
\sqcap^L &=& (\xymatrix{
A \ar[r]^-{1 \otimes \nu} &
A \otimes A \ar[r]^-{1 \otimes \Delta} &
A\otimes A \otimes A \ar[r]^-{\mu^{\mathsf{op}} \otimes 1} &
A\otimes A \ar[r]^-{\epsilon \otimes 1} &
A}) \\
\overline \sqcap^L &=& (\xymatrix{
A \ar[r]^-{1 \otimes \nu} &
A \otimes A \ar[r]^-{1 \otimes \Delta} &
A\otimes A \otimes A \ar[r]^-{\mu \otimes 1} &
A\otimes A \ar[r]^-{\epsilon \otimes 1} &
A}).
\end{array}
$$
The coinciding images of $\sqcap^R$ and $\overline \sqcap^R$ form a subalgebra of $A$ that we term the {\em base algebra}. Also the coinciding images of $\sqcap^L$ and $\overline \sqcap^L$ form a subalgebra of $A$ whose elements commute with the elements of the base algebra. The pair $\sqcap^R$ and $\overline \sqcap^L$, and also the pair $\sqcap^L$ and $\overline \sqcap^R$ restrict to mutually inverse anti-isomorphisms between these commuting subalgebras of $A$.
For more details we refer to \cite{BNSz}.

A {\em weak Hopf algebra} is a weak bialgebra $(A,(\mu,\nu),(\Delta,\epsilon))$ which admits a further linear map $\sigma:A\to A$ --- the so-called {\em antipode} --- rendering commutative the following diagrams.
$$
\xymatrix{
A \otimes A \ar[r]^-{1\otimes \sigma} &
A \otimes A \ar[d]^-\mu \\
A \ar[u]^-\Delta \ar[r]_-{\sqcap^L} &
A}
\qquad \qquad 
\xymatrix{
A \otimes A \ar[r]^-{\sigma \otimes 1} &
A \otimes A \ar[d]^-\mu \\
A \ar[u]^-\Delta \ar[r]_-{\sqcap^R} &
A}
$$
Whenever the antipode exists, it is unique, and it is an algebra anti-homomorphism as well as a coalgebra anti-homomorphism.

On elements $h$ of $A$, the Sweedler type index notation $\Delta(h)=h_1 \otimes h_2$ will be used, where implicit summation is understood. 
The multiplication is denoted by juxtaposition of elements.
The image of the number $1$ in $k$ under $\nu$ is also denoted by $1\in A$.

A {\em separable algebra} over a field $k$ is a monoid $R$ equipped with a separability structure in the monoidal category $\mathsf{vec}$ of $k$-vector spaces. In this case the separability structure can be given by the so-called {\em separability element} $\sum_i e_i \otimes f_i \in R \otimes R$ which is the image of the unit element of $R$ under the bilinear section of the multiplication. Then, by its bilinearity, the section sends any $r\in R$ to $\sum_i re_i \otimes f_i = \sum_i e_i \otimes f_i r$. Since it is a section of the multiplication, $\sum_i e_i f_i=1$. A {\em Frobenius-separability structure} on an algebra $R$ is a separability structure admitting a (necessarily unique) linear map $\psi:R \to k$ --- the {\em Frobenius functional} --- such that $\sum_i \psi(e_i)f_i=1= \sum_i e_i\psi(f_i)$.
For example, the base algebra $R$ of a weak bialgebra possesses a Frobenius-separable structure. The separability element is $1_1 \otimes \sqcap^R(1_2)$ and the Frobenius functional is the restriction of the counit to the base algebra. The {\em enveloping algebra} $R\otimes R^{\mathsf{op}}$ of a Frobenius-separable algebra $R$ is again Frobenius-separable via the obvious factorwise structure.

For {\em any} algebra $R$, a monoid in the monoidal category $\mathsf{bim}(R)$ of $R$-bimodules can be characterized, equivalently, as an algebra $A$ together with an algebra homomorphism $\eta:R \to A$. If $R$ is a separable algebra, then any separability structure on a monoid $(A,\eta)$ in $\mathsf{bim}(R)$
determines a separability structure on the algebra $A$ (as a monoid in $\mathsf{vec}$). Conversely, also any separability structure on the algebra $A$ determines a separability structure on the corresponding monoid $(A,\eta)$ in $\mathsf{bim}(R)$ for any algebra homomorphism $\eta:R \to A$ (although the correspondence is not bijective in general).

It was shown in \cite{B-GT-LC} that for any Frobenius-separable algebra $R$ the category $\mathsf{bim}(R\otimes R^{\mathsf{op}})$ of $R\otimes R^{\mathsf{op}}$-bimodules possesses a duoidal structure. In terms of a Frobenius-separability element $\sum_i e_i \otimes f_i\in R \otimes R$, the $\bullet$-monoidal product of any $R\otimes R^{\mathsf{op}}$-bimodules $M$ and $N$ lives on the subspace $\sum_{i,j} M(e_i\otimes f_j) \otimes (f_i \otimes e_j)N$ of the vector space $M\otimes N$. The $\bullet$-monoidal unit $j$ is $R\otimes R^{\mathsf{op}}$ with the actions provided by the multiplication:
\begin{equation} \label{eq:j_Rebim}
(x\otimes y)(p\otimes q)(x'\otimes y')=xpx'\otimes y'qy.
\end{equation}
The $\circ$-monoidal product of any $R\otimes R^{\mathsf{op}}$-bimodules $M$ and $N$ lives on the subspace $\sum_{i,j} (e_i \otimes 1)M(e_j\otimes 1) \otimes (1 \otimes f_i)N(1\otimes f_j)$ of the vector space $M\otimes N$. The $\circ$-monoidal unit $i$ also lives on $R\otimes R^{\mathsf{op}}$; at this time the actions are given in terms of the Frobenius functional $\psi:R \to k$ as
\begin{equation} \label{eq:i_Rebim}
(x\otimes y)(p\otimes q)(x'\otimes y')= y'px'\otimes xq \sum_i\psi(yf_i)e_i.
\end{equation}
As explained in \cite[Section 5.3]{BohmLack}, this duoidal category $\mathsf{bim}(R\otimes R^{\mathsf{op}})$ arises in the way described in Section \ref{sec:B(M,M)}. That is, it is the endohom category of a naturally Frobenius map monoidale $R\otimes R^{\mathsf{op}}$ in the monoidal bicategory of algebras, bimodules and bimodule maps.

The bimonoids in this duoidal category $\mathsf{bim}(R\otimes R^{\mathsf{op}})$ were identified in \cite{B-GT-LC} with the weak bialgebras whose base algebra is isomorphic to $R$. By \cite[Section 8.4]{BohmLack}, a weak bialgebra is a weak Hopf algebra if and only if the the corresponding bimonoid in $\mathsf{bim}(R\otimes R^{\mathsf{op}})$ is a Hopf monoid.

In a bit more detail, the bimonoid $a$ in $\mathsf{bim}(R\otimes R^{\mathsf{op}})$, associated to a weak bialgebra $(A,(\mu,\nu),(\Delta, \epsilon))$ with the base algebra $R$ (identified with $\sqcap^R(A)\subseteq A$), lives on the $R\otimes R^{\mathsf{op}}$-bimodule $A$ with the actions
\begin{equation} \label{eq:A_Rebim}
(x\otimes y)h(x'\otimes y') =x \overline \sqcap^L(y)h x' \overline \sqcap^L(y').
\end{equation}
The unit of the $\bullet$-monoid $a$ is 
\begin{equation} \label{eq:wba_unit}
\eta: j \to a, \qquad
p\otimes q \mapsto p \overline \sqcap^L(q).
\end{equation}
The multiplication occurs in the factorization of the algebra multiplication $\mu:A\otimes A \to A$ via the canonical epimorphism $A \otimes A \twoheadrightarrow A \bullet A$:
$$
\xymatrix{
A \otimes A \ar[rr]^-\mu \ar@{->>}[rd] && A \\
& A\bullet A \ar@{-->}[ru]}
$$
The counit of the $\circ$-comonoid $a$ is the map
\begin{equation} \label{eq:wba_counit}
\varepsilon: A \to i, \qquad
h \mapsto \sqcap^R(h_1) \otimes \overline \sqcap^R(h_2).
\end{equation}
The comultiplication is the composite of  the coalgebra comultiplication $\Delta:A \to A \otimes A$ and the canonical epimorphism $A \otimes A \twoheadrightarrow A \circ A$:
\begin{equation} \label{eq:wba_coprod}
\delta=(\xymatrix{
A \ar[r]^-\Delta &
A \otimes A \ar@{->>}[r] &
A \circ A}).
\end{equation}
If $A$ is a weak Hopf algebra, then the antipode of the Hopf monoid $a$ in $\mathsf{bim}(R\otimes R^{\mathsf{op}})$ is the same map $A\to A$ as the antipode of the weak Hopf algebra $A$.

After all these considerations, from Theorem \ref{thm:Maschke} we re-obtain \cite[Theorem 3.13]{BNSz} as follows.
\begin{theorem} \label{thm:wha_int}
For a weak Hopf algebra $A$ over a field $k$, with base algebra $R$, the following assertions are equivalent.
\begin{itemize}
\item[{(i')}]
$A$ is a separable $k$-algebra.
\item[{(i)}]
The algebra $A$ and the algebra homomorphism \eqref{eq:wba_unit} describe a separable monoid in $(\mathsf{bim}(R\otimes R^{\mathsf{op}}),\bullet,j)$.
\item[{(ii')}]
There is a normalized left integral in the sense of \cite[Definition 3.1]{BNSz}. That is, an element $t'$ of $A$ satisfying the following conditions.
\begin{itemize}
\item[$\bullet$] $ht'=\sqcap^L(h)t'$ in $A$, for all $h\in A$.
\item[$\bullet$] $\overline \sqcap^R(t')=1$ in $R$.
\end{itemize}
\item[{(ii)}] Regarding $A$ as a bimonoid $a$ in $\mathsf{bim}(R\otimes R^{\mathsf{op}})$, it possesses a normalized left integral in the sense of Definition \ref{def:integral}. That is, there exists an element $t$ of $A$ satisfying the conditions of part (ii') together with the further condition
\begin{itemize}
\item[$\bullet$] $t\sqcap^L(x)=t\, \overline \sqcap^R\overline \sqcap^L(x)$ in $A$, for all $x\in R$.
\end{itemize}
\item[{(iii')}]
There is a normalized right integral in the sense of \cite[Definition 3.1]{BNSz}. That is, an element $t'$ of $A$ satisfying the following conditions.
\begin{itemize}
\item[$\bullet$] $t'h=t'\sqcap^R(h)$ in $A$, for all $h\in A$.
\item[$\bullet$] $ \sqcap^R(t')=1$ in $R$.
\end{itemize}
\item[{(iii)}] Regarding $A$ as a bimonoid $a$ in $\mathsf{bim}(R\otimes R^{\mathsf{op}})$, it possesses a normalized right integral in the sense of Definition \ref{def:integral}. That is, there exists an element $t$ of $A$ satisfying the conditions of part (iii') together with the further condition
\begin{itemize}
\item[$\bullet$] $\overline \sqcap^L(x) t=\sqcap^R\sqcap^L(x)t$ in $A$, for all $x\in R$.
\end{itemize}
\end{itemize}
\end{theorem}

\begin{proof}
We have (i) $\Leftrightarrow$ (ii) $\Leftrightarrow$ (iii) by Theorem \ref{thm:Maschke}. Since $R$ is a separable $k$-algebra --- see \cite[Proposition 2.11]{BNSz} --- so is $R\otimes R^{\mathsf{op}}$. Therefore (i) $\Leftrightarrow$ (i'). 
Finally, (i') $\Leftrightarrow$ (ii') $\Leftrightarrow$ (iii') by \cite[Theorem 3.13]{BNSz} and its symmetric counterpart.
\end{proof}

Note that any integral $t$ in part (ii) of Theorem \ref{thm:wha_int} can be used as $t'$ in part (ii'). Conversely, if $t'$ is an integral in part (ii'), then $t'1_1 \sqcap^L\sqcap^R(1_2)$ is a suitable $t$ in part (ii).
Similarly, any integral $t$ in part (iii) can be used as $t'$ in part (iii'). Conversely, if $t'$ is an integral in part (iii'), then $\sqcap^R\sqcap^L(1_1)1_2t'$ is a suitable $t$ in part (iii).

Analogously, from Theorem \ref{thm:DualMaschke} we obtain the following.
\begin{theorem} \label{thm:wha_coint}
For a weak Hopf algebra $(A,(\mu,\nu),(\Delta,\epsilon),\sigma)$ over a field $k$, with base algebra $R$, the following assertions are equivalent.
\begin{itemize}
\item[{(i')}]
$(A,\Delta,\epsilon)$ is a coseparable $k$-coalgebra.
\item[{(i)}]
$(A,\delta, \varepsilon)$ of \eqref{eq:wba_coprod} and \eqref{eq:wba_counit} is a coseparable comonoid in $(\mathsf{bim}(R\otimes R^{\mathsf{op}}),\circ, i)$.
\item[{(ii')}]
There is a normalized left cointegral in the sense of a linear map $\tau':A \to k$ satisfying the following conditions.
\begin{itemize}
\item[$\bullet$] $h_1 \tau'(h_2)=\sqcap^L(h_1)\tau'(h_2)$ in $A$, for all $h\in A$.
\item[$\bullet$] $\tau'\, . \sqcap^L=\epsilon$.
\end{itemize}
\item[{(ii)}] Regarding $A$ as a bimonoid $a$ in $\mathsf{bim}(R\otimes R^{\mathsf{op}})$, it possesses a normalized left cointegral in the sense of Definition \ref{def:cointegral}. That is, there exists a linear map $\tau:A\to k$ satisfying the conditions of part (ii') together with the further condition
\begin{itemize}
\item[$\bullet$] $\tau(xh)=\tau(h\sqcap^R\sqcap^L(x))$, for all $x\in R$ and $h\in A$.
\end{itemize}
\item[{(iii')}]
There is a normalized right cointegral in the sense of a linear map $\tau':A \to k$ satisfying the following conditions.
\begin{itemize}
\item[$\bullet$] $\tau'(h_1)h_2=\tau'(h_1)\sqcap^R(h_2)$ in $A$, for all $h\in A$.
\item[$\bullet$] $\tau'. \sqcap^R=\epsilon$.
\end{itemize}
\item[{(iii)}] Regarding $A$ as a bimonoid $a$ in $\mathsf{bim}(R\otimes R^{\mathsf{op}})$, it possesses a normalized right cointegral in the sense of Definition \ref{def:cointegral}. That is, there exists a linear map $\tau:A \to k$ satisfying the conditions of part (iii') together with the further condition
\begin{itemize}
\item[$\bullet$] $\tau(h\overline \sqcap^L(x))=\tau(\sqcap^L(x)h)$, for all $x\in R$ and $h\in A$.
\end{itemize}
\end{itemize}
\end{theorem}

Note that any cointegral $\tau$ in part (ii) of Theorem \ref{thm:wha_coint} can be used as $\tau'$ in part (ii'). Conversely, if $\tau'$ is a cointegral in part (ii'), then $\tau'(1_1-\sqcap^R(1_2))$ is a suitable $\tau$ in part (ii).
Similarly, any cointegral $\tau$ in part (iii) can be used as $\tau'$ in part (iii'). Conversely, if $\tau'$ is a cointegral in part (iii'), then $\tau'(\sqcap^L(1_1)-1_2)$ is a suitable $\tau$ in part (iii).

\subsection{Hopf algebroids over central base algebras}

In \cite[Example 6.18]{AguiarMahajan} the category $\mathsf{bim}(R)$ of bimodules of a {\em commutative} algebra $R$ (say, over a field $k$) was shown to carry a duoidal structure as follows. 
The $\circ$-monoidal product is the usual $R$-bimodule tensor product provided by the coequalizer
$$
\xymatrix{
M \otimes R \otimes N  
\ar@<2pt>[rrr]^-{\textrm{right $R$-action }\otimes 1} 
\ar@<-2pt>[rrr]_-{1\otimes \textrm{ left $R$-action}} &&&
M \otimes N \ar[r] &
M \circ N}
$$
for any $R$-bimodules $M$ and $N$ and for the tensor product $\otimes$ of vector spaces. Thus the $\circ$-monoidal unit $i$ is $R$ with equal left and right actions provided by the multiplication.
Since $R$ is commutative, $R$-bimodules can be identified with modules over the commutative algebra $R\otimes R$. The $\bullet$-monoidal product of $R$-bimodules $M$ and $N$ is their $R\otimes R$-module tensor product occurring in the coequalizer
$$
\xymatrix{
M \otimes R \otimes R \otimes N  
\ar@<2pt>[rrr]^-{\textrm{$(R\otimes R)$-action }\otimes 1} 
\ar@<-2pt>[rrr]_-{1\otimes \textrm{ $(R\otimes R)$-action}} &&&
M \otimes N \ar[r] &
M \bullet N.}
$$
The $\bullet$-monoidal unit $j$ is $R\otimes R$ with the actions
$
x(p\otimes q)y=xp \otimes yq.
$
As explained in \cite[Section 5.2]{BohmLack}, this duoidal category $\mathsf{bim}(R)$ arises in the way described in Section \ref{sec:B(M,M)}. That is, it is the endohom category of a naturally Frobenius map monoidale $R$ in the monoidal bicategory of algebras, bimodules and bimodule maps.

In \cite[Example 6.44]{AguiarMahajan} the bimonoids in this duoidal category were identified with a particular kind of Takeuchi bialgebroid in \cite{Takeuchi} (see also \cite{Lu} and \cite{KadisonSzlachanyi}). The occurring bialgebroids are distinguished by the property that their source and target maps below land in the center of the total algebra.
Explicitly, a monoid in $(\mathsf{bim}(R),\bullet,j)$ is characterized by a $k$-algebra $A$ --- termed the {\em total algebra} --- and an algebra homomorphism $\eta$ from $R\otimes R$ to the center of $A$. Such a homomorphism $\eta$ is conveniently encoded in a pair of algebra homomorphisms $s:=\eta(-\otimes 1)$ --- the {\em source map} --- and $t:=\eta(1\otimes -)$ --- the {\em target map} --- from $R$ to the center of $A$ (so that $\eta(x\otimes y)=s(x)t(y)$). 
A comonoid in $(\mathsf{bim}(R),\circ,i)$ consists of an $R$-bimodule $A$ with $R$-bimodule maps $\delta:A \to A\circ A$ and $\varepsilon:A \to R$ such that $\delta$ is coassociative with the counit $\varepsilon$.
The bimonoid compatibility axioms translate to the conditions that the $R$-actions corresponding to the comonoid structure of $A$ are 
\begin{equation} \label{eq:A_Rbim}
xhy=s(x)t(y)h, \qquad
\forall x,y\in R,\ h\in A,
\end{equation}
and that $\delta:A \to A\circ A$ and $\varepsilon:A \to R$ are algebra homomorphisms (with respect to the well-defined factorwise multiplication on $A\circ A$).
Such a datum can be seen as a left bialgeboid in the sense of \cite{KadisonSzlachanyi} and, interchanging the roles of the source and target maps, also as a right bialgebroid in the sense of \cite{KadisonSzlachanyi}.
 
By \cite[Section 8.3]{BohmLack} an antipode (in the sense of \cite[Theorem 7.2]{BohmLack}) for such a bimonoid in $\mathsf{bim}(R)$ is a linear map $\sigma:A \to A$ satisfying 
$$
\begin{array}{ll}
\sigma(s(x)h)=t(x)\sigma(h) \qquad
& \sigma(t(x)h)=s(x)\sigma(h) \\
h_1\sigma(h_2)=s(\varepsilon(h)) \qquad
&\sigma(h_1)h_2=t(\varepsilon(h))
\end{array}
$$
for all $h\in A$ and $x\in R$, where a Sweedler type index notation $\delta(h)=h_1\circ h_2$ is used, with implicit summation understood; and juxtaposition stands for the multiplication of elements in the $k$-algebra $A$.
This structure is termed a {\em Hopf algebroid $A$ with central base algebra $R$}, because it can be seen as particular instance of a Hopf algebroid in the sense of \cite[Definition 2.2]{Hgd_int} (see also \cite{BSz:hgd} for a slightly more restrictive case) --- with left bialgeboid structure $(A,R,s,t,\delta,\varepsilon)$, right bialgeboid structure $(A,R,t,s,\delta,\varepsilon)$ and antipode $\sigma$.
If not only the algebra $R$ but also $A$ is commutative then this reduces to a groupoid object in the opposite of the category of commutative algebras and their homomorphisms, discussed in Appendix A.1 of \cite{Ravenel}.

Summarizing the above facts, from Theorem \ref{thm:Maschke} we obtain the following.

\begin{theorem} \label{thm:hgd_Maschke}
For any commutative algebra $R$ and any Hopf algebroid $A$ with central base algebra $R$, the following assertions --- formulated using the notation of this section --- are equivalent.
\begin{itemize}
\item[{(i)}] 
With the multiplication $A \bullet A \to A$ and the unit $R\otimes R \to A$, $x\otimes y \mapsto s(x)t(y)$, $A$ is a separable monoid in $(\mathsf{bim}(R),\bullet,j)$.
\item[{(ii)}] There is a normalized left integral in $A$ in the sense of Definition \ref{def:integral}. That is, an element $n$ of $A$ satisfying the following conditions.
\begin{itemize}
\item[$\bullet$] For all $h\in A$, $hn-s(\varepsilon(h))n$ belongs to the ideal $\{s(x)k-t(x)k\vert x\in R,\ k\in A\}$ in $A$.
\item[$\bullet$] $\varepsilon(n)=1$.
\end{itemize}
\item[{(iii)}] There is a normalized right integral in $A$ in the sense of Definition \ref{def:integral}. That is, an element $n$ of $A$ satisfying the following conditions.
\begin{itemize}
\item[$\bullet$] For all $h\in A$, $nh-s(\varepsilon(h))n$ belongs to the ideal $\{s(x)k-t(x)k\vert x\in R,\ k\in A\}$ in $A$.
\item[$\bullet$] $\varepsilon(n)=1$.
\end{itemize}
\end{itemize}
\end{theorem}

The equivalent assertions of  \cite[Theorem 3.1]{Hgd_int} imply the assertions of Theorem \ref{thm:hgd_Maschke} --- separability over $R$ implies separability over $R\otimes R$ --- but not conversely: the Hopf algebroid $R\otimes R$ in \cite[Example 3.1]{Lu} satisfies the assertions of Theorem \ref{thm:hgd_Maschke} but not those of \cite[Theorem 3.1]{Hgd_int} (for arbitrary commutative algebras $R$).

On the contrary, from Theorem \ref{thm:DualMaschke} we re-obtain \cite[Theorem 3.2]{Hgd_int} for Hopf algebroids with central base algebra as follows.

\begin{theorem} \label{thm:hgd_DualMaschke}
For any commutative algebra $R$ and any Hopf algebroid $A$ with central base algebra $R$, the following assertions --- formulated using the notation of this section --- are equivalent.
\begin{itemize}
\item[{(i)}] 
The comonoid $(A,\delta,\varepsilon)$ in $(\mathsf{bim}(R),\circ,i)$ --- where $A$ is understood to be an $R$-bimodule as in \eqref{eq:A_Rbim} --- is coseparable. 
\item[{(ii)}] There is a normalized left cointegral in the coinciding senses of Definition \ref{def:cointegral} and \cite[Definition 2.9]{Hgd_int}. That is, a linear map $\nu:A\to R$ subject to the  following conditions.
\begin{itemize}
\item[$\bullet$] $\nu(s(x)h)=x\nu(h)$ in $R$, for all $x\in R$ and $h\in A$.
\item[$\bullet$] $h_1 t(\nu(h_2))=s(\nu(h))$ in $A$, for all $h\in A$.
\item[$\bullet$] $\nu(1_A)=1_R$.
\end{itemize}
\item[{(iii)}] 
There is a normalized right cointegral in the coinciding senses of Definition \ref{def:cointegral} and \cite[Definition 2.9]{Hgd_int}. That is, a
linear map $\nu:A\to R$ subject to the  following conditions.
\begin{itemize}
\item[$\bullet$] $\nu(t(x)h)=x\nu(h)$ in $R$, for all $x\in R$ and $h\in A$.
\item[$\bullet$] $s(\nu(h_1))h_2=t(\nu(h))$ in $A$, for all $h\in A$.
\item[$\bullet$] $\nu(1_A)=1_R$.
\end{itemize}
\end{itemize}
\end{theorem}

\subsection{Hopf monads on autonomous monoidal categories}

In the bicategory $\mathsf{prof}$ the 0-cells are the (small) categories, the 1-cells $\mathsf A \to \mathsf B$ are the {\em profunctors} --- that is, the functors $\mathsf B^{\mathsf{op}} \times \mathsf A \to \mathsf{set}$ --- and the 2-cells are the natural transformations. The horizontal composition of any 1-cells $F:\mathsf A \to \mathsf B$ and $G:\mathsf B \to \mathsf C$ is given by the coend $\int^{b\in \mathsf B^0} F(-,b) \times G(b,-)$.
This bicategory $\mathsf{prof}$ is monoidal via the cartesian product of categories.
Any functor $f:\mathsf A \to \mathsf B$ can be regarded as a left adjoint profunctor $\mathsf B(-,f-)$; that is, as a map in $\mathsf{prof}$.
In this way any monoidal category $(\mathsf C,\otimes,K)$ can be seen as a map monoidale in $\mathsf{prof}$. It is furthermore naturally Frobenius if and only if $(\mathsf C,\otimes,K)$ is {\em autonomous}; that is, every object has a left and a right dual, see \cite{LopezFranco,LF-S-W}.
In this case there is a duoidal structure on $\mathsf{prof}(\mathsf C,\mathsf C)$ as in Section \ref{sec:B(M,M)}.

Via the above inclusion of $\mathsf{cat}$ into $\mathsf{prof}$, any monoidal comonad $t$ on  a monoidal category $(\mathsf C,\otimes,K)$ gives rise to a monoidal comonad on the naturally Frobenius map monoidale $(\mathsf C,\otimes,K)$ in $\mathsf{prof}$; hence to a bimonoid in the duoidal category $\mathsf{prof}(\mathsf C,\mathsf C)$.
By \cite[Section 8.5]{BohmLack}, this bimonoid is a Hopf monoid
if and only if it is a Hopf comonad in the sense of \cite[Section 3.6]{BruguieresVirelizier}; equivalently, in the sense of \cite[Section 2.7]{BLV}.

While Theorem \ref{thm:Maschke} seems to have no interesting message in this case, from Theorem \ref{thm:DualMaschke} we re-obtain \cite[Theorem 6.5]{BruguieresVirelizier} (see also \cite[Theorem 8.11]{TuraevVirelizier}) as follows.

\begin{theorem}
For an autonomous monoidal category $(\mathsf C,\otimes,K)$, and a Hopf comonad $t$ on $(\mathsf C,\otimes,K)$ --- with comonad structure $(\delta,\varepsilon)$ and monoidal structure $(t_2,t_0)$ --- the following assertions are equivalent.
\begin{itemize}
\item[{(i')}] The comonoid $(t,\delta,\varepsilon)$ is coseparable in $\mathsf{cat}(\mathsf C,\mathsf C)$
(considered with the monoidal structure provided by the composition of functors).
\item[{(i)}] The comonoid $(t,\delta,\varepsilon)$ is coseparable in the monoidal category $(\mathsf{prof}(\mathsf C,\mathsf C),\circ,i)$.
\item[{(ii)}] There exists a normalized left cointegral in the coinciding senses of Definition \ref{def:cointegral} and \cite[Section 6.3]{BruguieresVirelizier}. That is, a morphism $\Lambda:t(K)\to K$ rendering commutative the following diagrams. 
$$
\xymatrix{
t(K) \ar[r]^-{\delta_K} \ar[d]_-\Lambda &
tt(K) \ar[d]^-{t(\Lambda)} \\
K \ar[r]_-{t_0} &
t(K)}
\qquad \qquad
\xymatrix{
K \ar[r]^-{t_0} \ar@{=}[rd] &
t(K) \ar[d]^-\Lambda \\
& K}
$$
\item[{(iii)}] There exists a normalized right cointegral in the sense of Definition \ref{def:cointegral}. That is, a retraction $\tau$ of the natural transformation
$$
\xymatrix{
\mathsf C(K,-) \ar[r]^-t &
\mathsf C(t(K),t(-)) \ar[r]^-{\mathsf C(t_0,1)} &
\mathsf C(K,t(-))}
$$
rendering commutative the following diagram.
$$
\xymatrix{
\mathsf C(K,t(-)) \ar[rr]^-{\mathsf C(K,\delta_{-})} \ar[d]_-\tau &&
\mathsf C(K,tt(-)) \ar[d]^-{\tau_{t(-)}} \\
\mathsf C(K,-) \ar[r]_-t &
\mathsf C(t(K),t(-)) \ar[r]_-{\mathsf C(t_0,1)} &
\mathsf C(K,t(-))}
$$
\end{itemize}
\end{theorem}

\subsection{Hopf categories}

For any braided monoidal small category $(\mathsf C,\otimes, K)$, categories enriched in the monoidal category of comonoids in $\mathsf C$ were studied in \cite{BCV:HopfCat}. This includes small categories themselves (that are re-obtained if $(\mathsf C,\otimes, K)$ is the cartesian monoidal category  $\mathsf{set}$ of sets).

In \cite{span|V}, a monoidal bicategory $\mathsf{span}\vert {\mathsf C}$ was associated to the same datum; that is, to a braided monoidal small category  $(\mathsf C,\otimes, K)$. Any set $X$ was shown to induce a naturally Frobenius map monoidale in the vertically opposite bicategory $(\mathsf{span}\vert {\mathsf C})^\vop$ --- obtained from $\mathsf{span}\vert {\mathsf C}$ by formally reversing the 2-cells ---; also denoted by the same symbol $X$.
Categories enriched in the category of comonoids in $\mathsf C$, with the object set $X$, were interpreted as opmonoidal monads in $(\mathsf{span}\vert {\mathsf C})^\vop$ on the naturally Frobenius map monoidale $X$ (with suitably chosen 1-cell parts; see below). Consequently, they can be seen as bimonoids in the duoidal endohom category $(\mathsf{span}\vert {\mathsf C})^\vop(X,X)=(\mathsf{span}\vert {\mathsf C})(X,X)^{\mathsf{op}}$.
By \cite[Paragraph 4.11]{span|V}, a category enriched in the category of comonoids in $\mathsf C$, with the object set $X$, is a {\em Hopf $\mathsf C$-category} in the sense of \cite[Definition 3.3]{BCV:HopfCat} if and only if the corresponding bimonoid in $(\mathsf{span}\vert {\mathsf C})^\vop(X,X)$ 
is a Hopf monoid. In particular, a Hopf $\mathsf{set}$-category is precisely a small groupoid.

In this section we apply Theorem \ref{thm:Maschke} and Theorem \ref{thm:DualMaschke} to Hopf $\mathsf C$-categories as  above.

Without repeating the description of the monoidal bicategory $(\mathsf{span}\vert {\mathsf C})^\vop$ and its naturally Frobenius map monoidale provided by a set $X$, we only present briefly the the resulting  duoidal endohom category $(\mathsf{span}\vert {\mathsf C})^\vop(X,X)$. An object in it consists of a span 
$\xymatrix@C=12pt{X & \ar[l] A \ar[r] & X}$ and a map (of sets) $a$ from $A$ to the set of objects of $\mathsf C$.
A morphism $(A,a) \to (A',a')$ consists of a map of spans $f:A' \to A$ and a family of morphisms in $\mathsf C$,
$\{\varphi_h: a'(h) \to af(h)\}$ labelled by the elements $h$ of the set $A'$. The composition of morphisms comes from the composition of maps and the composition in $\mathsf C$.

The $\circ$-monoidal product of any objects $(A,a)$ and $(B,b)$ is given by the pullback span
$$
\xymatrix@R=10pt{
&& B \circ A \ar[ld] \ar[rd] \\
& B \ar[ld] \ar[rd] && A  \ar[ld] \ar[rd] \\
X && X && X}
$$
and the map sending $(d,h)\in B \circ A$ to the object $b(d)\otimes a(h)$ of $\mathsf C$.
The $\circ$-product of morphisms $(f,\varphi):(A,a) \to (A',a')$ and $(g,\gamma):(B,b) \to (B',b')$ consists of the map of spans
$$
g\circ f: B' \circ A' \to B \circ A, \qquad
(d,h) \mapsto (g(d),f(h))
$$
between the pullback spans and the morphisms 
$$
\gamma_d \otimes \varphi_h: b'(d) \otimes a'(h) \to bg(d) \otimes af(h)
\textrm{ in } \mathsf C, \textrm{ for } (d,h) \in B' \circ A'.
$$ 
The $\circ$-monoidal unit $i$ consists of the trivial span
$\xymatrix@C=12pt{X & \ar@{=}[l] X \ar@{=}[r] & X}$
and the constant map sending each element of $X$ to the monoidal unit $K$ of $\mathsf C$.

The $\bullet$-monoidal product of any objects $(A,a)$ and $(B,b)$ is given by the product $B \bullet A$ in the category of spans and their maps; together with the map sending $(d,h)\in B \bullet A$ to the object $b(d)\otimes a(h)$ of $\mathsf C$.
The $\bullet$-product of morphisms $(f,\varphi):(A,a) \to (A',a')$ and $(g,\gamma):(B,b) \to (B',b')$ consists of the map of spans
$$
g\bullet f: B' \bullet A' \to B \bullet A, \qquad
(d,h) \mapsto (g(d),f(h))
$$
between the product spans and the morphisms 
$$
\gamma_d \otimes \varphi_h: b'(d) \otimes a'(h) \to bg(d) \otimes af(h)
\textrm{ in } \mathsf C, \textrm{ for } (d,h) \in B' \bullet A'.
$$ 
The $\bullet$-monoidal unit $j$ consists of the terminal (a.k.a. complete) span
$\xymatrix@C=15pt{X & \ar[l]_-{p_1} X^2 \ar[r]^{p_2} & X}$
and the constant map sending each element of the cartesian product set $X^2\equiv X \times X$ to the monoidal unit $K$ of $\mathsf C$.

A bimonoid in this duoidal category $(\mathsf{span}\vert {\mathsf C})^\vop(X,X)=(\mathsf{span}\vert {\mathsf C})(X,X)^{\mathsf{op}}$ consists of a span
$\xymatrix@C=12pt{X & \ar[l]_-t A \ar[r]^-s & X}$, 
a map $a$ from $A$ to the set of objects in $\mathsf C$ together with the following morphisms:
\begin{itemize}
\item comultiplication $(
\xymatrix@C=12pt{A\circ A \ar[r]^-{\displaystyle \cdot} & A}, 
\xymatrix@C=12pt{a(h) \otimes a(k) \ar[r]^-{\mu_{h,k}} & a(h.k)} )$
\item counit $(
\xymatrix@C=12pt{X \ar[r]^-u & A},
\xymatrix@C=12pt{K \ar[r]^-{\eta_x} & a(u(x))})$
\item multiplication $(
\xymatrix@C=42pt{A \ar[r]^-{h\mapsto (h_1,h_2)} & A \bullet A},
\xymatrix@C=12pt{a(h) \ar[r]^-{\delta_h} & a(h_1) \otimes a(h_2)})$
\item unit $(
\xymatrix@C=12pt{A \ar[r]^-{!} & X^2},
\xymatrix@C=12pt{a(h) \ar[r]^-{\varepsilon_h} & K})$
\end{itemize}
subject to suitable compatibility conditions.

Consider a category enriched in the category of comonoids in $\mathsf C$, with the object set $X$, with hom-comonoids $(a(x,y),\delta_{x,y},\varepsilon_{x,y})$, composition morphisms $\mu_{x,y,z}$ and unit morphisms $\eta_x$ (for $x,y,z\in X$). It gives rise to a bimonoid in $(\mathsf{span}\vert {\mathsf C})^\vop(X,X)$ with the underlying terminal span
$\xymatrix@C=12pt{X & \ar[l]_-{p_1} X^2 \ar[r]^-{p_2} & X}$, 
the map sending $(x,y)\in X^2$ to $a(x,y)$ together with the following morphisms:
\begin{itemize}
\item comultiplication $(
\xymatrix@C=12pt{X^2\circ X^2  \cong X^3 \ar[r]^-{p_{13}} & X^2}, 
\xymatrix@C=18pt{a(x,y) \otimes a(y,z) \ar[r]^-{\mu_{x,y,z}} & a(x,z)} )$
\item counit $(
\xymatrix@C=12pt{X \ar[r]^-\Delta & X^2},
\xymatrix@C=12pt{K \ar[r]^-{\eta_x} & a(x,x)})$
\item multiplication $(
\xymatrix@C=12pt{X^2 \ar[r]^-1 & X^2=X^2 \bullet X^2},
\xymatrix@C=12pt{a(x,y) \ar[r]^-{\delta_{x,y}} & a(x,y) \otimes a(x,y)})$
\item unit $(
\xymatrix@C=12pt{X^2 \ar[r]^-1 & X^2},
\xymatrix@C=12pt{a(x,y) \ar[r]^-{\varepsilon_{x,y}} & K})$.
\end{itemize}
The compatibility conditions translate precisely to the requirements that the composition and the unit morphisms of the enriched category must be comonoid morphisms.
 
By the above considerations Theorem \ref{thm:Maschke}  leads to the following.

\begin{theorem}
For a Hopf $\mathsf C$-category 
$(\{(a(x,y),\delta_{x,y},\varepsilon_{x,y}) \}_{(x,y)\in X^2},
\{\mu_{x,y,z} \}_{(x,y,z)\in X^3},$
$\{\eta_x\}_{x\in X})$, 
the following assertions are equivalent.
\begin{itemize}
\item[{(i)}] The comonoid $\{(a(x,y),\delta_{x,y},\varepsilon_{x,y}) \}_{(x,y)\in X^2}$ in $(\mathsf{span}\vert {\mathsf C}(X,X),\bullet, j)$ is coseparable.
\item[{(i')}] Each of the comonoids $(a(x,y),\delta_{x,y},\varepsilon_{x,y})$ in $({\mathsf C}, \otimes, K)$, for $(x,y)\in X^2$, is coseparable.
\item[{(ii)}] For all $x\in X$, $\eta_x:K \to a(x,x)$ has a left $a(x,x)$-comodule retraction. That is, for all $x\in X$ there is a morphism $\theta_x:a(x,x) \to K$ in $\mathsf C$ rendering commutative the following diagrams.
$$
\xymatrix{
a(x,x) \ar[r]^-{\theta_x} \ar[d]_-{\delta_{x,x}} &
K \ar[d]^-{\eta_x} \\
a(x,x) \otimes a(x,x) \ar[r]_-{1\otimes \theta_x} &
a(x,x)}
\qquad \qquad
\xymatrix@R=27pt{
K \ar[r]^-{\eta_x} \ar@{=}[rd] &
a(x,x) \ar[d]^-{\theta_x} \\
& K}
$$
\item[{(iii)}] For all $x\in X$, $\eta_x:K \to a(x,x)$ has a right $a(x,x)$-comodule retraction. That is, for all $x\in X$, there is a morphism $\theta_x:a(x,x) \to K$ in $\mathsf C$ rendering commutative the following diagrams.
$$
\xymatrix{
a(x,x) \ar[r]^-{\theta_x} \ar[d]_-{\delta_{x,x}} &
K \ar[d]^-{\eta_x} \\
a(x,x) \otimes a(x,x) \ar[r]_-{\theta_x \otimes 1} &
a(x,x)}
\qquad \qquad
\xymatrix@R=27pt{
K \ar[r]^-{\eta_x} \ar@{=}[rd] &
a(x,x) \ar[d]^-{\theta_x} \\
& K}
$$
\end{itemize}
\end{theorem}

Theorem \ref{thm:DualMaschke}, on the other hand, tells the following.

\begin{theorem}
For a Hopf $\mathsf C$-category 
$(\{(a(x,y),\delta_{x,y},\varepsilon_{x,y}) \}_{(x,y)\in X^2},
\{\mu_{x,y,z} \}_{(x,y,z)\in X^3},$
$\{\eta_x\}_{x\in X})$, 
the following assertions are equivalent.
\begin{itemize}
\item[{(i)}]
The monoid $(\{a(x,y)\},\{\mu_{x,y,z} \},\{\eta_x\})$ in $(\mathsf{span}\vert {\mathsf C}(X,X),\circ, i)$ admits a separability structure. That is, a family of morphisms $\{\partial_{x,y,z}:a(x,z) \to a(x,y) \otimes a(y,z)\}_{(x,y,z)\in X^3}$ making the following diagrams commute, for all $x,y,v,z\in X$.
$$
\xymatrix@C=-4pt@R=12pt{
a(x,y) \otimes a(y,z) \ar[rr]^-{\partial_{x,v,y} \otimes 1} \ar[dd]_-{1\otimes \partial_{y,v,z}} \ar[rd]_-{\mu_{x,y,z}} &&
a(x,v) \otimes a(v,y) \otimes a(y,z) \ar[dd]^-{1\otimes \mu_{v,y,z}} \\
& a(x,z) \ar[rd]^-{\partial_{x,v,z}} \\
a(x,y) \otimes a(y,v) \otimes a(v,z) \ar[rr]_-{\mu_{x,y,v} \otimes 1} &&
a(x,v) \otimes a(v,z)}
\quad
\xymatrix@R=41pt{
a(x,y) \ar[r]^-{\partial_{x,v,y}} \ar@{=}[rd] &
a(x,v) \otimes a(v,y)  \ar[d]^-{\mu_{x,v,y}} \\
& a(x,y).}
$$
\item[{(ii)}] There exists a normalized left cointegral in the sense of Definition \ref{def:cointegral}; equivalently, a normalized {\em left integral family} in the sense of \cite[Definition 4.1]{BFVV:LarsonSweedler} with an additional normalization condition. That is, a family of morphisms in $\mathsf C$, 
$\{\theta_{x,y}:K \to a(x,y)\}_{x,y\in X}$ rendering commutative the following diagrams for all $x,y,z\in X$.
$$
\xymatrix@C=40pt{
a(x,y) \ar[r]^-{1\otimes \theta_{y,z}} \ar[d]_-{\varepsilon_{x,y}} &
a(x,y)  \otimes a(y,z) \ar[d]^-{\mu_{x,y,z}} \\
K \ar[r]_-{\theta_{x,z}} & 
a(x,z)} 
\qquad \qquad
\xymatrix@R=27pt{
K \ar[r]^-{\theta_{x,y}} \ar@{=}[rd] &
a(x,y) \ar[d]^-{\varepsilon_{x,y}} \\
& K}
$$
\item[{(iii)}] There exists a normalized right cointegral in the sense of Definition \ref{def:cointegral}; equivalently, a normalized {\em right integral family} in the sense of \cite[eq. (35)]{BFVV:LarsonSweedler} with an additional normalization condition. That is, a family of morphisms in $\mathsf C$, 
$\{\theta_{x,y}:K \to a(x,y)\}_{x,y\in X}$ rendering commutative the following diagrams for all $x,y,z\in X$.
$$
\xymatrix@C=40pt{
a(y,z) \ar[r]^-{\theta_{x,y} \otimes 1} \ar[d]_-{\varepsilon_{y,z}} &
a(x,y)  \otimes a(y,z) \ar[d]^-{\mu_{x,y,z}} \\
K \ar[r]_-{\theta_{x,z}} & 
a(x,z)} 
\qquad \qquad
\xymatrix@R=27pt{
K \ar[r]^-{\theta_{x,y}} \ar@{=}[rd] &
a(x,y) \ar[d]^-{\varepsilon_{x,y}} \\
& K}
$$
\end{itemize}
\end{theorem}


\bibliographystyle{plain}

\end{document}